\documentclass[12pt]{amsart}

 \usepackage{a4wide}
 \usepackage{amssymb}
  \makeatletter
  \CheckCommand*\refstepcounter[1]{\stepcounter{#1}%
      \protected@edef\@currentlabel
       {\csname p@#1\endcsname\csname the#1\endcsname}%
  }
  \renewcommand*\refstepcounter[1]{\stepcounter{#1}%
    \protected@edef\@currentlabel
      {\csname p@#1\expandafter\endcsname\csname the#1\endcsname}%
  }
  \def\labelformat#1{\expandafter\def\csname p@#1\endcsname##1}
  \DeclareRobustCommand\Ref[1]{\protected@edef\@tempa{\ref{#1}}%
     \expandafter\MakeUppercase\@tempa
  }
  \makeatother

  \makeatletter
  \newcommand{\numberlike}[2]{%
     \expandafter\def\csname c@#1\endcsname{%
         \expandafter\csname c@#2\endcsname}%
  }
  \makeatother

  \def\DefaultNumberTheoremWithin{section}

  \theoremstyle{plain}
  \newtheorem{lemma}{Lemma}
     \numberwithin{lemma}{\DefaultNumberTheoremWithin}
     \labelformat{lemma}{Lemma~#1}
  
     \numberwithin{claim}{\DefaultNumberTheoremWithin}
     \numberlike{claim}{lemma}
     \labelformat{claim}{Claim~#1}
  \newtheorem{theorem}{Theorem}
     \numberwithin{theorem}{\DefaultNumberTheoremWithin}
     \numberlike{theorem}{lemma}
     \labelformat{theorem}{Theorem~#1}
  \newtheorem{corollary}{Corollary}
     \numberwithin{corollary}{\DefaultNumberTheoremWithin}
     \numberlike{corollary}{lemma}
     \labelformat{corollary}{Corollary~#1}
  \newtheorem{proposition}{Proposition}
     \numberwithin{proposition}{\DefaultNumberTheoremWithin}
     \numberlike{proposition}{lemma}
     \labelformat{proposition}{Proposition~#1}
  
     \numberwithin{conjecture}{\DefaultNumberTheoremWithin}
  \numberlike{Conjecture}{lemma}
     \labelformat{conjecture}{Conjecture~#1}
  \theoremstyle{definition}
  \newtheorem{definition}{Definition}
     \numberwithin{definition}{\DefaultNumberTheoremWithin}
     \numberlike{definition}{lemma}
     \labelformat{definition}{Definition~#1}

  \theoremstyle{definition}
  
     \numberwithin{question}{\DefaultNumberTheoremWithin}
     \numberlike{question}{lemma}
     \labelformat{question}{Question~#1}

  \theoremstyle{definition}
  
     \numberwithin{problem}{\DefaultNumberTheoremWithin}
     \numberlike{problem}{lemma}
     \labelformat{problem}{Problem~#1}

  \theoremstyle{remark}
  
     \numberwithin{remark}{\DefaultNumberTheoremWithin}
     \numberlike{remark}{lemma}
     \labelformat{remark}{Remark~#1}
  \theoremstyle{remark}

  \newtheorem{example}{Example}
     \numberwithin{example}{\DefaultNumberTheoremWithin}
     \numberlike{example}{lemma}
     \labelformat{example}{Example~#1}
  
     \labelformat{case}{Case~#1}
     \numberwithin{case}{lemma}
  
     \labelformat{step}{Step~#1}
     \numberwithin{step}{lemma}

  \labelformat{equation}{(#1)}
  \labelformat{figure}{Figure~#1}
  \labelformat{section}{Section~#1}
  \labelformat{subsection}{Section~#1}
  \labelformat{subsubsection}{\S#1}

  \newcommand{\NN}{\mathbb{N}}
  \newcommand{\KK}{\mathbb{K}}
  \newcommand{\ZZ}{\mathbb{Z}}
  \newcommand{\RR}{\mathbb{R}}

  \newcommand{\Sub}{\mathrm{Sub}}
	\newcommand{\depth}{\mathrm{depth}}

	 \newcommand{\PP}{\mathbb{P}}

    \newcommand{\Tor}{\mathrm{Tor}}

	\newcommand{\pdim}{\mathrm{pdim}}

		\newcommand{\reg}{\mathrm{reg}}

\newcommand{\st}{\mathrm{star}}

\newcommand{\lk}{\mathrm{lk}}

\newcommand{\sd}{\mathrm{sd}}
\newcommand{\fe}{\mathfrak{e}}
\newcommand{\fu}{\mathfrak{u}}
\newcommand{\ffi}{\mathfrak{i}}
\newcommand{\supp}{\mathrm{supp}}
\newcommand{\AST}{\mathop{\ast}}

\begin{document}

\title{Asymptotic syzygies of Stanley-Reisner rings of iterated subdivisions}

\author{Aldo Conca}
\author{Martina Juhnke-Kubitzke}
\author{Volkmar Welker}

\maketitle
\begin{abstract}
Inspired by recent results of Ein, Lazarsfeld,  Erman and Zhou on the non-vanishing of Betti numbers of high Veronese subrings, we describe the behaviour of the Betti numbers of  Stanley-Reisner rings associated with  iterated barycentric or  edgewise subdivisions of a given simplicial complex. Our results show that for a simplicial complex $\Delta$ of dimension $d-1$ and for $1\leq j\leq d-1$ the number of $0$'s   the $j$\textsuperscript{th} linear strand of the minimal free resolution of the $r$\textsuperscript{th} barycentric or edgewise subdivision is bounded above only in terms of $d$ and $j$ (and independently of $r$). 
\end{abstract}

\section{Introduction}
Ein, Lazarsfeld and Erman   \cite{EinLazarsfeld, EinErmanLazarsfeld},  and Zhou \cite{Zhou} studied recently  the asymptotic behavior of the syzygies of algebraic varieties under high Veronese embeddings.  In particular, they treated the case of  the syzygies of $r$\textsuperscript{th} Veronese embeddings $v_{r}(\PP^n)$ of projective space $\PP^n$.   Roughly speaking, they proved that   for large $r$  the syzygies of $v_{r}(\PP^n)$ are non-zero for most of  the homological positions and internal degrees that are allowed by the restrictions  imposed by the projective dimension and by the Castelnuovo-Mumford regularity. Similar results, but with less precise bounds, are obtained for arithmetically Cohen-Macaulay varieties and are conjectured in general. 

The goal of this paper is to prove that a similar behavior occurs also for the syzygies of  Stanley-Reisner rings  of iterated barycentric subdivisions and edgewise subdivisions. These two combinatorial operations on simplical complexes have some formal similarity (but also some important dissimilarity) to the formation of Veronese subalgebras.  Moreover, as we explain later on,  the edgewise subdivision of a simplicial complex is closely related, via  Gr\"obner deformations,   with the formation of Veronese subalgebras of the associated Stanley-Reisner ring. 

Let $\KK$ be a field and  let $S=\KK[x_1,\dots,x_n]$ be the polynomial ring.  Let $I$ be a homogeneous ideal of $S$ such that $I\subset (x_1,\dots,x_n)^2$.  Let  $A =S/I=\bigoplus_{i \geq 0} A_i$  be a standard graded $\KK$-algebra. Denote by $\beta_{i,j}(A)$ the graded Betti number of $A$, i.e.,
$$\beta_{i,j}(A)=\dim_\KK \Tor_i^S(A,\KK)_j=\dim_\KK H_i(m_A,A)_j.$$
Here $H_i(m_A,A)$ denotes the $i$\textsuperscript{th} Koszul homology of the maximal homogeneous ideal $m_A=\bigoplus_{i \geq 1} A_i$ of $A$  and the index $j$ on the right always denotes the selection of the $j$\textsuperscript{th} homogeneous component. 

The $r$\textsuperscript{th} Veronese algebra of $A$ is by definition 
  $$A^{(r)} = \bigoplus_{i \geq 0} A_{ir}.$$

One of the main results of  \cite{EinLazarsfeld} asserts that for every integer $j$ in the interval $[1,n-1]$ the graded Betti numbers $\beta_{i,i+j}(S^{(r)})$  of the $r$\textsuperscript{th} Veronese subalgebra of the polynomial ring $S$ 
  are asymptotically (i.e., for large $r$)   
  non-zero for  every integer   $i$ in an interval $[a_j,b_j]$ with specified endpoints $a_j,b_j$  that depend on $j$. Note that the Castelnuovo-Mumford regularity of  $S^{(r)}$ is always  $\leq n-1$ with equality  for large $r$.  Hence it is clear that $\beta_{i,i+j}(S^{(r)})=0$ for $j$ outside the interval  $[1,n-1]$ with the exception of $\beta_{0,0}(S^{(r)})=1$.  Comparing the size of  the  intervals $[a_j,b_j]$ with the projective dimension of $S^{(r)}$ (i.e., the length of the minimal free resolution) one deduces that for large $r$ the Betti number  $\beta_{i,i+j}(S^{(r)})$ is non-zero for  most of  the  values of $i,j$, which are allowed by the restrictions imposed by the value of the projective dimension and the Castelnuovo-Mumford regularity. 
    
 In  \cite{EinErmanLazarsfeld}  a similar result, but with a less precise description of the intervals,  is proved  for the Betti numbers  $\beta_{i,i+j}(A^{(r)})$ of the Veronese subalgebras of an arbitrary Cohen-Macaulay algebra $A$.   
 In particular, it follows that for every $j=1,\dots,\dim A-1$ one has: 
 \begin{equation}
 \label{EEL1}
      \lim_{r\rightarrow \infty}\frac{ \#\{i~:~\beta_{i,i+j}(A^{(r)})\neq 0\}}{\pdim (A^{(r)})}
        =1.
    \end{equation}
 Furthermore, the authors conjecture that the same behavior holds for an arbitrary standard graded $\KK$-algebra.

 Let  $\Delta$ be a simplicial complex and denote by  $\KK[\Delta]$ the corresponding Stanley-Reisner ring.   We will consider two combinatorial operations on simplicial complexes: the iterated barycentric subdivision and edgewise subdivisions. We will denote by   $\sd^r(\Delta)$ 
  the  $r$\textsuperscript{th} iterated barycentric subdivision of $\Delta$ and by $\Delta^{\langle r \rangle}$ the  $r$\textsuperscript{th}  edgewise subdivision  of $\Delta$,  whose definition is recalled in \ref{notat}. 
  
  We will study the asymptotic behavior of $\beta_{i,i+j}(\KK[ \sd^r(\Delta)])$ and $\beta_{i,i+j}(\KK[ \Delta^{\langle r \rangle} ])$. The main results we prove are  \ref{thm:asymptoticsBary},  \ref{prop:BarycentricHomology}, \ref{thm:asymptoticsEdge} and \ref{prop:edgewiseHomology} whose main content  is the following: 
  
  \begin{theorem}
  \label{maintheo}
   Let $\Delta$ be an arbitrary simplicial complex of dimension $d-1>0$. Let $\Delta(r)$ be either the iterated barycentric subdivision $\sd^r(\Delta)$ or the  edgewise subdivision $\Delta^{\langle r \rangle}$  of $\Delta$. Then for large $r$ the Castelnuovo-Mumford regularity of   $\KK[\Delta(r)]$ is given by: 
$$\reg(\KK[\Delta(r)])=
  \left\{\array{ll}  d-1, &  \mbox{if }  \widetilde{H}_{d-1}(\Delta;\KK)= 0 \\
   d,   &   \mbox{if }  \widetilde{H}_{d-1}(\Delta;\KK)\neq 0.  \endarray \right. $$
Furthermore:   \begin{itemize}
  \item[(1)]  For every $j=1,\dots,d-1$ one has that  $\#\{ i ~:~ \beta_{i,i+j}(\KK[\Delta(r) ])=0\}$ is bounded above in terms of $d,j$ (and independently of $r$). In particular: 
 \begin{equation*}
      \lim_{r\rightarrow \infty}\frac{ \#\{i~:~\beta_{i,i+j}(\KK[\Delta(r) ])\neq 0\}}{\pdim (\KK[\Delta(r) ])}
        =1.
    \end{equation*}
 \item[(2)]   If $\widetilde{H}_{d-1}(\Delta;\KK)\neq 0$ then 
  \begin{equation*}
      \lim_{r\rightarrow \infty}\frac{ \#\{i~:~\beta_{i,i+d}(\KK[\Delta(r)])\neq 0\}}{\pdim (\KK[\Delta(r)])}
    \end{equation*}
  is a rational number in the interval $[0,1)$ that  can be described in terms of the  minimal $(d-1)$-cycles of $\Delta$.  \end{itemize}    
     \end{theorem} 
 
The limit in (2) does not depend on whether one takes the iterated 
barycentric  subdivision or the edgewise subdivision and for any rational number in the interval $[0,1)$ and any $d$, we construct a $(d-1)$-dimensional Cohen-Macaulay complex that realizes the specified limit.

Brun and R\"omer proved in  \cite{BrunRoemer} that for every simplicial 
complex $\Delta$ there is  a term order $\prec$ such that the Veronese 
subalgebra $\KK[\Delta]^{(r)}$ of $\KK[\Delta]$ has a Gr\"obner deformation 
to $\KK[ \Delta^{\langle r \rangle}]$. Betti numbers can only increase under a Gr\"obner deformation and hence one has
$$\beta_{i,i+j}(\KK[\Delta]^{(r)} )\leq \beta_{i,i+j}(\KK[\Delta^{\langle r  \rangle}] )$$
for every $r,i,j$. 
In particular, the assertion about the limit in   \ref{maintheo} (1) for $\Delta^{\langle r\rangle}$ 
when $\Delta$ is Cohen-Macaulay can be deduced from the result  of Ein, Erman and Lazarsfeld  
Equation \ref{EEL1} that will be presented in \cite{EinErmanLazarsfeld}.
Note however that our  assertion  in \ref{maintheo} about the cardinality of $ \{ i: \beta_{i,i+j}(\KK[\Delta(r) ])=0\}$ being bounded only in terms of $j,d$ is  simply not true for iterated Veronese as can be deduced from the results in  \cite{EinLazarsfeld, EinErmanLazarsfeld}. Another 
important difference between our results and the ones in  \cite{EinLazarsfeld, EinErmanLazarsfeld} is the statement about the $d$\textsuperscript{th} strand.  In our context the limit behavior gives a number in $[0,1)$ and any number is indeed possible. In the Veronese case, at least for the coordinate rings of smooth varieties in characteristic $0$,  the length on the $d$\textsuperscript{th} strand is constant and hence the corresponding limit is $0$, see \cite[Eq.(1.3) pg.607]{EinLazarsfeld}.   
 
Our methods are mostly geometric and combinatorial and are based on  Hochster's formula expressing the Betti numbers 
in terms of homologies of induced subcomplexes. The case when
$\Delta = \Delta_{d-1}$ is the full $(d-1)$-simplex turns out to be the
crucial case. A key observation
for this analysis is (see \ref{lem:InnerFace}) that links of faces of
$\Delta_{d-1}^{\langle r \rangle}$, whose interior lies in the interior of $\Delta_{d-1}$,
are barycentric subdivisions of boundaries of full simplices.
This essentially reduces the analysis of the Betti numbers of $\KK[\Delta^{\langle r \rangle}]$  and of $\KK[\sd^r(\Delta)]$ to the
analysis of the Betti numbers of $\KK[\sd (\Delta_{d-1})]$. This analysis is performed in \ref{sec:bary}
and appears to be of independent interest.  

We remark that many of the arguments that we present go through for arbitrary subdivision operations that satisfy mild 
assumptions. The actual formulation of the assumption is technical and does not
give too much insight, but the requirement is twofold. First, one has to  deal with a subdivision operator that  can be   applied   iteratively. The second technical
requirement roughly says that when iteratively subdividing a simplex 
``sufficiently many'' new vertices are created in the interior of the simplex.
The conclusion will be that for a simplicial complex $\Delta$ of dimension $d-1$ and
a suitable subdivision operation $\Sub$, if for some $j<d$, $1\leq k$ and $r \geq 0$ 
we have $\beta_{k,k+j}(\KK[\Sub^r(\Delta)]) \neq 0$ then 

$$\lim_{r\rightarrow \infty}\frac{ \#\{i~:~\beta_{i,i+j}(\KK[\Sub^r(\Delta) ])\neq 0\}}{\pdim (\KK[\Sub^r(\Delta) ])}
  = 1.$$

Examples of suitable subivision operations can be found in various  articles, see for example \cite{CMS} and \cite{W}.
Note that there are suitable subdivision operations that fail to produce non-zero 
$\beta_{i,i+j}(\Sub(\KK[\Delta]))$ for every  $1 \leq j \leq d-1$ and for some $i$. As an example serves the
subdivision of $(d-1)$-dimensional simplicial complexes where we subdivide by coning the
boundary of each $(d-1)$-simplex over a point in the interior of the simplex.

  \subsection{Notation and background}
\label{notat}
  Let $\Delta$ be a simplicial complex on ground set $\Omega$. We call an element $F \in \Delta$ a \emph{face}
  of $\Delta$ with \emph{dimension} $\dim F := \# F -1$. The \emph{dimension} $\dim \Delta$ of $\Delta$ is then
  the maximal dimension of one of its faces. If $\dim \Delta = d-1$, then we
  write $f^\Delta = (f_{-1}^\Delta,\ldots, f_{d-1}^\Delta)$ for the \emph{$f$-vector} of $\Delta$, where
  $f_i^\Delta$ counts the number of $i$-dimensional faces in $\Delta$.

  Sometimes we  are
  interested in subsets of $\Delta$ that are not necessarily simplicial complexes themselves.
  Let $\Gamma \subseteq \Delta$ be such a subset. Then we consider $\Gamma$ as a partially ordered
  set ordered by inclusion and write $\Delta(\Gamma)$ for its \emph{order complex} that is the set of
  all subsets of $\Gamma$ that are linearly ordered. If $\Gamma = \Delta \setminus \{ \emptyset\}$, then
  $\sd(\Delta) := \Delta(\Gamma)$ is the \emph{barycentric subdivision} of $\Delta$.
  The other subdivision operation that is important for this paper is the edgewise subdivision.
  Assume $\Delta$ is $(d-1)$-dimensional with vertex set $[n] := \{1,\ldots, n\}$ and
  let $r\geq 1$ be a positive integer. Set $\Omega_{r,n}:=\{(i_1,\ldots,i_n)\in \NN^n~:~i_1+\cdots+i_n=r\}$.
  Denote by $\fe_i$ the $i$\textsuperscript{th} unit vector of $\RR^n$. By the obvious
  identification, we can consider $\Delta$ as a simplicial complex over the
  vertex set $\Omega_{1,n}=\{\fe_1,\ldots,\fe_n\}$.
  For $i\in [n]$ set $\fu_i:=\fe_i+\cdots+\fe_n$ and
  for $a=(a_1,\ldots,a_n)\in\ZZ^n$  let
  $\ffi(a):=\sum_{l=1}^n a_l\cdot \fu_l$. The \emph{$r$\textsuperscript{th} edgewise
  subdivision} of $\Delta$ is the simplicial complex $\Delta^{\langle r \rangle}$ on ground
  set $\Omega_{r,n}$ such that $F \subseteq \Omega_{r,n}$ is a simplex in $\Delta^{\langle r \rangle}$
  if and only if

  \begin{enumerate}
    \item[(i)]$\bigcup_{a\in F}\supp(a)\in\Delta$, and,
    \item[(ii)] for all $a, \tilde{a}\in F$ either
        $\ffi(a-\tilde{a})\in\{0,1\}^n$ or
        $\ffi(\tilde{a}-a)\in\{0,1\}^n$.
  \end{enumerate}

  If we denote by $|\Delta|$ the geometric realization of $\Delta$, then we can choose
  realizations such that $|\Delta| = |\sd(\Delta)| = |\Delta^{\langle r \rangle}|$.
  If $F$ is a face of $\Delta$, we sometimes write $|F|$ to denote the geometric realization of
  the subcomplex $2^F$ of $\Delta$. By $\partial |\Delta|$ we denote the boundary of $|\Delta|$.

  For a simplical complex $\Delta$ on ground set $\Omega$ its \emph{Stanley-Reisner ideal} $I_\Delta$ is
  the ideal in $S = \KK[x_\omega~ :~ \omega \in \Omega]$ generated by monomials $\prod_{\omega \in N} x_\omega$
  for $N \not\in \Delta$. $\KK[\Delta] := S/I_\Delta$ is called the \emph{Stanley-Reisner ring} of $\Delta$.
  For background on the algebraic invariants of $\KK[\Delta]$ studied in this paper we refer to
  \cite{BH,Eisenbud}.
  
 \section{Basic algebraic invariants under Veronese, barycentric and edgewise subdivision} \label{sec:VeroBaryEdge}
 
 In this section we recall how some basic invariants of standard graded $\KK$-algebras and Stanley-Reisner rings behave  
 under the Veronese operation  on the algebra and barycentric/edgewise subdivision on the simplicial complex. 
 
 In the following $A=\bigoplus_{i=0}^\infty A_i$ will denote a standard graded $\KK$-algebra of Krull dimension $d>0$ and $H^i_{m_A}(A)$ will denote the $i$\textsuperscript{th} local cohomology module with respect to the maximal homogeneous ideal $m_A=\bigoplus_{i=1}^\infty A_i$ of $A$.  Furthermore,   $\Delta$ will denote a  simplicial complex of dimension $d-1\geq 0$.   We denote by $t_1(A)$ the largest degree of a minimal generator of the defining ideal of $A$ as a quotient of the polynomial ring in $\dim A_1$ variables.  Hence $t_1(\KK[\Delta])$ is the largest cardinality of a minimal non-face of $\Delta$ and $t_1(\KK[\Delta])=2$ if and only if $\Delta$ is a flag complex. 
 Set $a_d(A)=\max \{ j ~:~ H_{m_A}^d(A)_j\neq 0\}$ and  
 $$w(\Delta,\KK)=\left\{\array{ll}  d-1, &  \mbox{if }  \widetilde{H}_{d-1}(\Delta;\KK)= 0 \\ d,   &   \mbox{if }  \widetilde{H}_{d-1}(\Delta;\KK)\neq 0.  \endarray \right.$$
 We have: 
 
 \begin{table}[htdp]
\caption{}
\begin{center}
\begin{tabular}{|c|c|c|c|c|}  
\hline
                &  (1)   &  (2) &  (3)  &   (4)  \\ \hline 
                &  $A^{(r)}$  &   $\KK[\Delta]^{(r)}$ &   $\KK[\Delta^{\langle r\rangle}]$ &    $\KK[\sd^r (\Delta)]$ \\ \hline 
$\dim $    & $\dim A$ & $\dim  \KK[\Delta]$ &   $\dim \KK[\Delta]$   & $\dim \KK[\Delta]$  \\  \hline 
$\depth$  & $\geq \depth\  A$ &   $\depth\  \KK[\Delta]$   &  $\depth\  \KK[\Delta]$ & $\depth\  \KK[\Delta]$  \\  \hline 
$t_1$       &  $2$  $^*$& $2$  $^*$  &      $\left\{\array{ll} t_1( \KK[\Delta])-1  \\   t_1(\KK[\Delta]) \endarray \right.$   $^{**}$  &  $2$      \\  \hline 
$\reg$     &  
$\left\{\array{ll}  d-1, &  \mbox{if }   a_d(A)<0 \\ d,   &   \mbox{if }  a_d(A)\geq 0 \endarray \right.$  $^*$  & 
$w(\Delta,\KK)$  $^*$  & 
$w(\Delta,\KK)$  $^*$  &
$w(\Delta,\KK)$  
    \\  \hline 
 \end{tabular}
\end{center}
\label{How}
\end{table}

\noindent where $^*$ means that the formula holds for large values of $r$.  The formulas marked with $^*$ can often be made more precise, for example, 
$t_1(A^{(r)})\leq \max\{ 2, \lceil  t_1(A)/r\rceil \}$ holds for every $r$.  

The data in  column (1)  are obtained by applying the  formula that relates the local cohomology before and after applying the Veronese functor, see \cite[Thm. 3.1.1]{GW},  and the assertion that $H_{m_A}^d(A)_j\neq 0$ for every $j\leq a_d(A)$, see  \cite[Prop. 2.2]{CJR}.  It might happen that  $\depth\ A^{(r)}>\depth\ A$ for every $r>1$. Take, for example,  $A=\KK[x,y]/(x^2,xy)$ or, if one wants a domain,  $A=\KK[x^4,x^3y,xy^3,y^4]$.   

The data in column (2) are obtained using \cite[Thm. 3.1.1]{GW} and Hochster's formula for local cohomology modules of Stanley-Reisner rings \cite[Thm. 5.3.8]{BH}. 

Concerning the data in column (3), the dimension is clear by construction. That  depth is invariant under edgewise subdivision follows from Munkres' result \cite[Thm. 3.1]{Munkres2} since   $|\Delta|  = |\Delta^{\langle r \rangle}|$.  Furthermore, $^{**}$ holds if  $\Delta$ is not a simplex.   If  $\Delta$ is a simplex and $d-1\geq 1$, then  $t_1(\KK[\Delta^{\langle r\rangle}])=2$ for every $r>1$. In  \ref{lem:minGen} we prove the statement concerning the maximal degree of a minimal generator and also specifies when one has $t_1(\KK[\Delta^{\langle r\rangle}])=t_1( \KK[\Delta])-1$ and  $t_1(\KK[\Delta^{\langle r\rangle}])=t_1(\KK[\Delta])$, respectively.  The formula for the regularity will be proved  in \ref{cor:reg}.  

Finally, concerning the data in column (4), the dimension and the value of $t_1$ are clear by construction.  The formula for regularity follows from  \cite[Prop. 2.6]{KW-multiplicity} and that the depth is invariant follows  again by Munkres' result \cite[Thm. 3.1]{Munkres2} since  $|\Delta| = |\sd(\Delta)|$ or directly from \cite[Cor. 2.5]{KW-multiplicity}.

\section{Barycentric subdivision} \label{sec:bary}

  In this section we provide the analysis of the Betti numbers
  $\beta_{i,i+j}(\KK[\sd(\Delta)])$ of the Stanley-Reisner ring of
  the barycentric subdivision of a simplicial complex $\Delta$. Besides its
  combinatorial appeal it is crucial for the proof of our main result \ref{edgewise:Betti}.

  We have recorded already in Table \ref{How}  how basic  algebraic invariants behave under barycentric subdivision.
   Let us record furthermore how the  projective dimension changes under  barycentric subdivision: 
 
 $$\pdim (\KK[\sd(\Delta])) =\pdim( \KK[\Delta])+\sum_{i\geq 1}f_i^{\Delta}.$$
    
Recall Hochster's formula for the graded Betti numbers of  Stanley-Reisner rings 
\cite[Thm. 5.5.1]{BH}:
 \begin{equation}
  \label{hochster}
  \beta_{i,i+j}(\KK[\Delta])=\sum_{\substack{W\subseteq [n]\\ \# W=i+j}}\dim_\KK \widetilde{H}_{j-1}(\Delta_W;\KK),
 \end{equation}
 where $\Delta_W = \{ F \in \Delta~:~F \subseteq W\}$.
In particular, if $\Delta$ is a simplicial complex on vertex set $\Omega$, 
\begin{equation*}
 \beta_{i,i+j}(\KK[\Delta])\neq 0\qquad \Leftrightarrow \qquad \exists W\subseteq \Omega,\;\# W=i+j \mbox{ such that } \widetilde{H}_{j-1}(\Delta_W;\KK)\neq 0.
\end{equation*}

We are interested in the range of the different strands in the minimal
resolution of $\KK[\sd(\Delta)]$. More precisely, for all $1\leq j\leq
\reg (\KK[\sd(\Delta)])$, we would like to identify  the set  of the $i$'s  such that 
$\beta_{i,i+j}(\KK[\sd(\Delta)]) \neq 0$ in terms of invariants of $\Delta$.

\subsection{Betti numbers for barycentric subdivisions of simplices}

We start our analysis with the study of the Betti numbers of the barycentric subdivision of the $(d-1)$-dimensional simplex $\Delta_{d-1}$. 
Note that in this case one has: 

$$\pdim (\KK[\sd(\Delta_{d-1}])) =2^d-d-1.$$

We start with a definition: 
\begin{definition}
\label{defmj} 
Given integers $d$ and $j$ such that $d \geq 1$ and $1 \leq j \leq d-1$ we set
$$m_j(d)=\left\{ \begin{array}{ll}
j,  & \mbox{ if }  j\leq d/2\\
2^{a+2} (c+d-j)-2d+j,  &  \mbox{ if } j\geq d/2 \mbox{ where } (2j-d)=a(d-j)+c \\  & \mbox{ with } a,c\in \NN  \mbox{ and } 0\leq c< d-j.
\end{array} 
\right.
$$ 
When  $d$ is clear from the context, we will suppress it from the notation and simply use $m_j$ for $m_j(d)$.  
\end{definition} 
 
 These numbers play an important role in the following results and, as we will see in the proofs, they arise by considerations related to the search of 
$(j-1)$-spheres as subcomplexes of $\sd(\Delta_{d-1})$ induced by as few vertices as possible.

\begin{theorem}\label{thm:Betti:bar}
  Let $d \geq 1$.  Let $\Delta_{d-1}$ be the $(d-1)$-simplex  and  let   $\sd(\Delta_{d-1})$ be its barycentric subdivision. Then:
  \begin{itemize}
    \item[(i)] If $1\leq j\leq \frac{d}{2}$, then
      $$
        \beta_{i,i+j}(\KK[\sd(\Delta_{d-1})])
        \begin{cases}
                 =0 \mbox{ for } 0\leq i\leq j-1,\\
           \neq 0 \mbox{ for } j\leq i\leq 2^d-d-1-m_{d-j-1},  \\
                 =0 \mbox{ for } 2^d-2d+j<i\leq 2^d-d-1.
        \end{cases}$$
    \item[(ii)] If $\frac{d}{2}<j\leq d-2$, then
      $$\beta_{i,i+j}(\KK[\sd(\Delta_{d-1})])
      \begin{cases}
                =0 \mbox{ for } 0\leq i \leq  j-1,  \\
         \neq 0 \mbox{ for } m_j\leq i \leq 2^d-2d+j,\\
             =0 \mbox{ for  } 2^d-2d+j<i \leq 2^d-d-1.
      \end{cases}$$
    \item[(iii)] $\beta_{i,i+d-1}(\KK[\sd(\Delta_{d-1})])\neq 0$ if and only if $i=2^d -d-1$.
  \end{itemize}
\end{theorem}

The theorem identifies whether  $\beta_{i,i+j}(\KK[\sd(\Delta_{d-1})])$ is zero or not, except for the cases
\begin{itemize}
\item $1\leq j\leq \frac{d}{2}$ and $2^d-d-m_{d-j-1}\leq i\leq 2^d-2d+j$, and,
\item $\frac{d}{2}<j\leq d-2$ and $j\leq i\leq m_j-1$.
\end{itemize}

 We now formulate some crucial lemmas and propositions, which will lead to a proof of \ref{thm:Betti:bar}. First, we introduce some notation that will be frequently used throughout
this section. For $1\leq j\leq d-1$ we set:

\begin{equation*}
  l_j(d-1)=\min\{i~:~\beta_{i,i+j}(\KK[\sd(\Delta_{d-1})])\neq 0\}
\end{equation*}
and
\begin{equation*} 
u_j(d-1)=\max\{i~:~\beta_{i,i+j}(\KK[\sd(\Delta_{d-1})])\neq 0\}.
\end{equation*}

\noindent We determine  bounds for $l_j(d-1)$ and $u_j(d-1)$ by constructing
$(j-1)$-spheres as subcomplexes of $\sd(\Delta_{d-1})$ induced
by as few vertices as possible.

In general we can construct induced $(j-1)$-spheres of $\sd(\Delta_{d-1})$ in the
following way. Let $(i_1,\ldots,i_r)\in \NN^r$ such that
\begin{eqnarray}
 \label{e1} i_1+\cdots +i_r+(r-1) & = & j-1, \\
 \label{e2} i_1+\cdots +i_r+2r & \leq & d.
\end{eqnarray}
 
 Let $\mathbb{S}^n$ denote the $n$-dimensional sphere. Note, that \ref{e1} implies
\begin{equation*}
 \mathbb{S}^{j-1}\cong\mathbb{S}^{i_1}\ast\cdots \ast \mathbb{S}^{i_r},
\end{equation*}
 where $\ast$ denotes the join operator.
We set $W_{i_1}:=\{A~:~\emptyset\neq A\subsetneq [i_1+2]\}$ and for
$2 \leq \ell \leq r$ we let
\begin{equation}\label{e3}
W_{i_\ell} := \{A\cup [i_1+\cdots +i_{\ell-1}+2(\ell-1)]~:~\emptyset\neq A\subsetneq [i_1+\cdots+i_\ell+2\ell]\setminus[i_1+\cdots+i_{\ell-1}+2(\ell-1)]\}.
\end{equation}
As an abstract simplicial complex, the restriction $\sd(\Delta_{d-1})_{W_{i_\ell}}$
is isomorphic to the barycentric subdivision of the boundary of an
$(i_\ell+1)$-simplex. Hence, geometrically, $\sd(\Delta_{d-1})_{W_{i_\ell}}$ is an
$i_\ell$-sphere.
Moreover,
\begin{align}\label{hom1}
  \sd(\Delta_{d-1})_{\bigcup_{\ell=1}^r W_{i_\ell}}&=\AST_{\ell=1}^r \sd(\Delta_{d-1})_{W_{i_\ell}}\notag\\
  &\cong \AST_{\ell=1}^r\mathbb{S}^{i_\ell}=\mathbb{S}^{j-1}.
\end{align}
Note that 

$$\# \Big(\displaystyle{\bigcup_{\ell=1}^r} W_{i_\ell}\Big)=\displaystyle{\sum_{\ell=1}^r (2^{i_\ell+2}-2)}.$$

We observe the following: 

\begin{lemma}
\label{againmj}  The numbers $m_j$ defined in \ref{defmj}  satisfy  the following equality: 
\begin{equation} \label{eq:mj}
  m_j+j=\min\Big\{\sum_{\ell=1}^r (2^{i_\ell+2}-2)~:~  \begin{array}{l} 
  (i_1,\dots,i_r)\in \NN^r, \\
  i_1+\cdots +i_r+(r-1)=j-1, \\
  i_1+\cdots +i_r+2r\leq d 
  \end{array} \Big\}.
\end{equation}
\end{lemma} 
The proof of   \ref{againmj} will be given in the \ref{Appendix}. 
Hence we obtain: 

\begin{corollary}
\label{coromj} 
The number $m_j+j$ is an upper bound for the minimal cardinality
of a subset $W$ of the set of vertices of $\sd(\Delta_{d-1})$ such that $\widetilde{H}_{j-1}(\sd(\Delta_{d-1})_W;\KK) \neq 0$.
In particular, $m_j$ is an upper bound for $l_j(d-1)$ for every $j=1,\dots,d-1$. 
\end{corollary}

Our next aim is to improve the  bounds for  $l_j(d-1)$ and
to provide bounds for $u_j(d-1)$. To this end we need two lemmas.
As we have already observed in Table \ref{How}  the barycentric subdivision of any simplicial complex  is a flag complex; i.e., all its minimal non-faces are of size $2$. In particular: 
\begin{lemma}\label{lem:flag}
  Let $\Delta$ be a $(d-1)$-dimensional simplicial complex on vertex set $[n]$ and let
  $W\subseteq \Delta\setminus \{\emptyset\}$ be a
  subset of the vertex set of $\sd(\Delta)$. Then
  $\sd(\Delta)_W$ is a flag complex or a full simplex
\end{lemma}

We prove the next lemma by adapting arguments used for the proof
of \cite[Lem. 2.1.14]{Gal}.

\begin{lemma}
  \label{lowerbound}
  Let $z$ be a $j$-cycle in a $(d-1)$-dimensional flag simplicial
  complex $\Delta$, which is not a boundary.
  Then the simplices in the support of $z$ contain at least $2(j+1)$
  vertices.
\end{lemma}
\begin{proof}
  We proceed by induction on $j$. If $z$ is  $0$-cycle, then there are
  at least $2$ simplices in its support.
  If $j=1$, then there exist at least $3$ vertices in the simplices in
  the support of $z$. If $z$ contains exactly $3$ vertices in its
  support, $\Delta$ cannot be flag.
  Let $j\geq 2$. Assume that $z$ is a $j$-cycle, which is not a
  boundary, and  assume that the set of vertices in the simplices in the
  support of $z$ is minimal. Then in the simplices in the support of
  $z$, there exist two vertices $v$ and $w$ that are not connected by an
  edge in $\Delta$. Otherwise, the vertices in the simplices from the
  support of $z$ form a simplex and hence $z$ is a boundary.
  Let $z=\displaystyle{\sum_{\sigma\in\Delta}}a_\sigma \sigma$.
  We write $z=vz_1+z_2$, where
  $$z_1=\displaystyle{\sum_{\substack{v\in \sigma\in \Delta\\a_{\sigma}\neq0}}}a_{\sigma}
   (\sigma\setminus\{v\}) \mbox{ and }z_2=\displaystyle{\sum_{\substack{v\notin\sigma\in\Delta
   \\ a_{\sigma}\neq 0}}}a_{\sigma}\sigma.$$
  It follows that $z_1$ has to be a $(j-1)$-cycle since otherwise
  $\partial z\neq 0$. Moreover, $z_1$ cannot be a boundary since
  otherwise the set of vertices in the simplices in the support of $z$
  is not minimal. Indeed, if $z_1=\partial z_3$, then $z_3+z_2$
  is a $j$-cycle, which is not a boundary. In particular, this
  contradicts the minimality of $z$. We conclude by induction, that the
  simplices in the support of $z_1$ contain at least $2j$ vertices.
  Since the simplices in the support of $z$ contain the two additional
  vertices $v$ and $w$, there are at least $2(j+1)$ vertices in the
  simplices in the support of $z$.
\end{proof}

Now using Hochster's formula \ref{hochster}, \ref{coromj}, \ref{lem:flag} and \ref{lowerbound},
we obtain the following bounds on $l_j(d-1)$ and $u_j(d-1)$.

\begin{proposition} 
  \label{le:barysimplex} ~
  
  \begin{itemize}
    \item[(i)] For $1\leq j\leq \frac{d}{2}$ one has $l_j(d-1)=j.$ 
    \item[(ii)]  For $\frac{d}{2}<j\leq d-2$ one has  $j\leq l_j(d-1)\leq m_j.$
    \item[(iii)] For $1\leq j\leq \frac{d}{2}$ one has $ 2^d-d-1-m_{d-j-1}\leq u_j(d-1) \leq 2^d-2d+j.$
    \item[(iv)] For $\frac{d}{2}<j\leq d-2$ one has $u_j(d-1) = 2^d-2d+j$.
    \item[(v)] $l_{d-1}(d-1)=u_{d-1}(d-1)=2^{d}-d-1$.
  \end{itemize}
\end{proposition}
\begin{proof}
  Since $\sd(\Delta_{d-1})$ is a cone over the barycentric subdivision of the boundary of $\Delta_{d-1}$,
   which is a triangulation of a sphere, it follows that $\sd(\Delta_{d-1})$ is Gorenstein.
   From graded Poincar\'e duality on the Koszul homology of Gorenstein rings (see \cite[Thm. 3.4.5]{BH})
   we deduce that
   \begin{equation}\label{Gorenstein}
   \beta_{i,i+j}(\KK[\sd(\Delta_{d-1})]) =
   \beta_{2^d-d-1-i,2^d-2-i-j}(\KK[\sd(\Delta_{d-1})]).
   \end{equation}
   Hence assertions (i) and (iv) are equivalent, and similarly assertions (ii) and (iii) are equivalent. We will now show (i) and (ii).

 By  \ref{coromj}  we   know that 
  $l_j(d-1) \leq m_j$ and hence in particular $l_j(d-1) \leq j$ if $1\leq j\leq \frac{d}{2}$. To show equality in (i)
  note that by \ref{lem:flag} all induced
  subcomplexes of $\sd(\Delta_{d-1})$ are flag. Hence we can infer from
  \ref{lowerbound} that a
  $(j-1)$-cycle of $\sd(\Delta_{d-1})$ is supported on at least $2j$
  vertices which implies by Hochster's formula \ref{hochster} that
  $\beta_{i,i+j}(\KK[\sd(\Delta_{d-1})]) = 0$
  for $1 \leq i \leq j-1$ and equality in (i) follows.
The same argument also shows the lower bound for $l_j(d-1)$ in (ii).

  Finally  (v) follows immediately  from \ref{Gorenstein} since $\beta_{i,i}(\KK[\sd(\Delta_{d-1})])=0$ for every $i>0$ and $\beta_{0,0}(\KK[\sd(\Delta_{d-1})])=1$. 
\end{proof}

After establishing the bounds on $l_j(d-1)$ and $u_j(d-1)$ we next turn
to a sequence of lemmas showing that there are no internal $0$s in the intervals we have identified.

\begin{lemma}\label{le:Betti:bar1}
  Let $1 \leq j \leq d-2$. Then we have
  $\beta_{i,i+j}(\KK[\sd(\Delta_{d-1})]) \neq 0$ for
  $$ 2^{j+1}-2 -j \leq i \leq 2^d - 2^{d-j}-1-j.$$
\end{lemma}

\begin{proof}
  Let $G$ be a $j$-dimensional face of $\Delta_{d-1}$ and consider
  the set $V_{<G}$ of vertices of $\sd(\Delta_{d-1})$ that correspond to
  faces of $\Delta_{d-1}$ properly contained in $G$. Then
  $\sd(\Delta_{d-1})_{V_{<G}}$ is the barycentric subdivision of the
  boundary of $G$ and hence triangulates a $(j-1)$-sphere and
  $\widetilde{H}_{j-1}(\sd(\Delta_{d-1})_{V_{<G}};\KK) \neq 0$.
  Let $F$ be an arbitrary face of $G$ of dimension $j-1$.
  Then $F$ is contained in the support of the homology
  $(j-1)$-cycle $z$ of the boundary of $G$.
  Let $V_F$ be the set of vertices in
  $\sd(\Delta_{d-1})$
  that correspond to faces of $\Delta_{d-1}$ that are neither
  subsets nor supersets of $F$. For any subset $W$ of $V_F$ the
  restriction of $\sd(\Delta_{d-1})$ to $V_{<G} \cup W$ has nonvanishing
  $(j-1)$\textsuperscript{st} reduced homology group. For this
  consider the image $\sd(z)$ of $z$ in the chain complex of
  $\sd(\Delta_{d-1})_{V_{<G} \cup W}$. Then $\sd(z)$ contains any
  $(j-1)$-simplex of $\sd(2^F)$ in its support. Each of these
  simplices contains the vertex $F$. But, if $\sd(z)$ were a boundary, then it could only be
  the boundary of a chain that contained a $j$-simplex with $F$ as
  a vertex. Since $F$ is $(j-1)$-dimensional, this simplex must also
  contain a vertex corresponding to a proper superset of $F$,
  but such a simplex does not exist in
  $\sd(\Delta_{d-1})_{V_{<G} \cup W}$.

  This and elementary counting shows that
  $\beta_{i,i+j}(\KK[\sd(\Delta_{d-1})]) \neq 0$
  for
  $$ 2^{j+1}-2 -j = \# V_{<G} -j \leq i \leq \# (V_{<G} \cup V_F)-j = 2^d - 2^{d-j}-1-j.$$
\end{proof}

\begin{proposition}\label{thm:continuous}
  Let $r\geq 2$ and $1 \leq j\leq d-2$. Let
 
  $(i_1,\ldots,i_r)\in \NN^r$  such that
  \begin{itemize}
    \item[(i)] $i_1+\cdots +i_r+(r-1)=j-1$.
    \item[(ii)] $i_1+\cdots +i_r+2r\leq d$.
  \end{itemize}
  Then $\beta_{i,i+j}(\KK[\sd(\Delta_{d-1})])\neq 0$ for all
  \begin{equation*}
    \sum_{\ell=1}^r (2^{i_{\ell}+2}-2)-j\leq i\leq
        \sum_{\ell=1}^r(2^{i_{\ell}+2}-2)+(2^{i_r+2}-2)2^{i_2+\ldots +i_{r-1}+2r-4}-j.
  \end{equation*}
\end{proposition}

\begin{proof}
We recall the definition of the sets $W_{i_\ell}$ from \ref{e3}.
  We set $W_{i_1}:=\{ A~:~\emptyset \neq A\subsetneq [i_1+2]\}$ and for $2\leq \ell\leq r$  we let
  \begin{equation*}
     W_{i_{\ell}}:=\{A\cup[i_1+\cdots +i_{\ell-1}+2(\ell-1)]~:~\emptyset \neq A\subsetneq
            [i_1+\cdots +i_\ell+2\ell]\setminus [i_1+\cdots +i_{\ell-1}+2(\ell-1)]\}.
  \end{equation*}
  We further define
  \begin{equation*}
    C(i_1,\ldots,i_r)=\Big\{A\cup B~:~ \substack{\emptyset\neq A\subsetneq
      [i_1+\cdots +i_r+2r]\setminus [i_1+\cdots +i_{r-1}+2(r-1)]\\
            B\subseteq [i_1+\cdots +i_{r-1}+2(r-1)]\setminus [i_1+2]}\Big\}.
  \end{equation*}
  First observe, that for $1\leq j\leq r-1$, $A \in W_{i_j}$ and $B\in C(i_1,\ldots,i_r)$ the set $\{A,B\}$ is
  a non-face of $\sd(\Delta_{d-1})$. Indeed, since
  $A\cap ([i_1+\cdots +i_r+2r]\setminus [i_1+\cdots +i_{r-1}+2(r-1)])=\emptyset$ but
  $B\cap ([i_1+\cdots +i_r+2r]\setminus [i_1+\cdots +i_{r-1}+2(r-1)])\neq \emptyset$,
  we could only have $A\subseteq B$.
  In this case, we arrive at a contradiction, since we know that $A\cap[i_1+2]\neq \emptyset$,
  but $B\cap [i_1+2]= \emptyset$.
  Hence, it follows that for any $D\subseteq C(i_1,\ldots,i_r)$ it holds that
  \begin{equation*}
    \sd(\Delta_{d-1})_{\bigcup_{\ell=1}^r W_{i_{\ell}}\cup D}=
    \sd(\Delta_{d-1})_{\bigcup_{\ell=1}^rW_{i_{\ell}}}\cup\sd(\Delta_{d-1})_{W_{i_r}\cup D}.
  \end{equation*}
  Moreover,
  \begin{equation*}
    \sd(\Delta_{d-1})_{\bigcup_{\ell=1}^r W_{i_{\ell}}}\cap\sd(\Delta_{d-1})_{W_{i_r}\cup D}=
    \sd(\Delta_{d-1})_{W_{i_r}}.
  \end{equation*}
  We claim that
  \begin{equation}\label{claim}
    \widetilde{H}_{\ell}(\sd(\Delta_{d-1})_{W_{i_r}\cup D};\KK)=
    \widetilde{H}_{\ell}(\sd(\Delta_{d-1})_{W_{i_r}};\KK)
  \end{equation}
for any $\ell\in\mathbb{N}$. 
  If we have shown \ref{claim},
  we can deduce from \ref{hom1} and the Mayer-Vietoris sequence in reduced homology for the above decomposition that
  \begin{equation*}
     \widetilde{H}_{\ell}(\sd(\Delta_{d-1})_{\bigcup_{\ell=1}^r W_{i_{\ell}}\cup D};\KK)=
       \begin{cases}
          \KK, & \mbox{ if }\ell=i_1+\cdots +i_r+(r-1)\\
          0,& \mbox{ otherwise}
       \end{cases}
  \end{equation*}
  for all $D\subseteq C(i_1,\ldots,i_r)$.
  This then in particular shows the assertion of the proposition since
  $\# C(i_1,\ldots,i_r)=(2^{i_r+2}-2)2^{i_2+\cdots +i_{r-1}+2r-4}$ and $\# \Big(\displaystyle{\bigcup_{\ell=1}^r} W_{i_\ell}\Big)=\displaystyle{\sum_{\ell=1}^r (2^{i_\ell+2}-2)}$.\\

  We now show claim \ref{claim}. The key idea is to interpret $\sd(\Delta_{d-1})_{W_{i_{r}}\cup D}$
  as the order complex of a poset.
  We distinguish two cases.

  \noindent {\sf Case 1:} $D=C(i_1,\ldots,i_r)$.

First note, that $W_{i_{r}}\cup C(i_1,\ldots,i_r)$  can be written in the following form
  \begin{equation*}
    W_{i_{r}}\cup C(i_1,\ldots,i_r)=\left\{A\cup B~:~\substack{\emptyset \neq A\subsetneq [i_1+\cdots+i_r+2r]
    \setminus [i_1+\cdots+i_{r-1}+2(r-1)]\\
    B\subseteq [i_1+\cdots +i_{r-1}+2(r-1)]\setminus [i_1+2] \mbox{ or } B=[i_1+\cdots+i_{r-1}+2(r-1)]}\right\}.
  \end{equation*}
To simplify notation, we set $P_1 = 2^{[i_1+\cdots +i_r+2r]\setminus [i_1+\cdots +i_{r-1}+2(r-1)]}-\{\emptyset, [i_1+\cdots +i_r+2r]\setminus [i_1+\cdots +i_{r-1}+2(r-1)]\}$ and
  $P_2 = 2^{[i_1+\cdots +i_{r-1}+2(r-1)]\setminus [i_1+2]}\cup \{[i_1+\cdots + i_{r-1}+2(r-1)]\}$. 
   Since $([i_1+\cdots+i_r+2r]\setminus
  [i_1+\cdots+i_{r-1}+2(r-1)])\cap [i_1+\cdots+i_{r-1}+2(r-1)]=\emptyset$
  any set $E\in W_{i_{r}}\cup C(i_1,\ldots,i_r)$ has a unique decomposition $E=E_1\cup E_2$, where
  $E_1\in P_1$ and
  $E_2 \in P_2$.
  In the following we consider $W_{i_{r}}\cup C(i_1,\ldots,i_r)$, $P_1$ and
  $P_2$ as partially ordered sets with order relation given by set inclusion.
  By the above arguments the map
  \begin{align*}
    \Phi: \left\{ \begin{array}{ccc} W_{i_r}\cup C(i_1,\ldots,i_r) &\rightarrow & P_1 \times P_2 \\
                                                           A\cup B &\mapsto & (A,B)
                   \end{array} \right.
  \end{align*}
  is well-defined. Here, for the two posets $P_1$ and $P_2$, we write $P_1 \times P_2$ for the
  partially ordered set on the Cartesian product with $(p_1,p_2) \leq (p_1',p_2')$ if and only
  if $p_s \leq p_s'$ in $P_s$ for $s \in \{1,2\}$.
  It is now straight forward to verify that $\Phi$ defines an isomorphism of partially ordered sets.
  Moreover, the poset $P_1 \times P_2$ is easily seen to be isomorphic to $P_1' \times P_2'$,
  where $P_1' = 2^{[i_r+2]}-\{\emptyset,[i_r+2]\}$, $P_2' = 2^{[i_2+\cdots+i_{r-1}+2r-4]}\cup \{\hat{1}\}$
  with order relation being inclusion and $\hat{1}$ being an artificial maximal element of $P_2'$.
  It then follows that the order complexes of $W_{i_r}\cup C(i_1,\ldots,i_r)$ and $P_1' \times P_2'$
  are isomorphic. Note that the order complex of $W_{i_r}\cup C(i_1,\ldots,i_r)$ is
  $\sd(\Delta_{d-1})_{W_{i_{r}}\cup C(i_1,\ldots,i_r)}$.
  We now look at the following map
  \begin{align*}
    f: \left\{ \begin{array}{ccc} P_1' \times P_2' & \rightarrow & P_1' \times P_2' \\
                                   (A,B) & \mapsto & (A,\hat{1}) \end{array} \right.
  \end{align*}
  Then $f$ is a poset map and satisfies $f^2((A,B)) = f((A,B)) \geq (A,B)$. Hence $f$ is a
  closure operator. Thus by \cite[Cor. 10.12]{Bjoerner}
  the order complexes of $P_1' \times P_2'$ and $f(P_1' \times P_2')$
  are homotopy equivalent. Since the projection on the first coordinate is an isomorphism of
  $f(P_1' \times P_2')$ and $P_1'$ it follows that the order complexes of $P_1' \times P_2'$ and
  $P_1'$ are homotopy equivalent.

  In particular, their homology groups are equal. Finally,
  since the order complex of $P_1'$ is the barycentric subdivision of the boundary complex
  of an $(i_r+1)$-simplex, we obtain

  \begin{equation*}
     \widetilde{H}_{\ell}(\sd(\Delta_{d-1})_{W_{i_r}\cup C(i_1,\ldots,i_r)};\KK) =
        \begin{cases}
           \KK, & \mbox{ if } \ell=i_r\\
             0, & \mbox{otherwise.}
        \end{cases}
   \end{equation*}
  Since by the discussion preceding \ref{hom1}, we have $\sd(\Delta_{d-1})_{W_{i_r}}\cong \mathbb{S}^{i_r}$, this concludes Case 1.

  \noindent {\sf Case 2:} $\emptyset \neq D\subsetneq C(i_1,\ldots,i_r)$.

  Using the map $\Phi$, defined in Case 1, and by an analogous argumentation as in this case, one sees
  that $W_{i_r}\cup D$ is isomorphic to a subposet $P_D$ of
  $P_1 \times P_2$ that contains $P_1 \times\{[i_1+\cdots+i_{r-1}+2(r-1)]\}$.
  For this note that $B\cup [i_1+\cdots +i_{r-1}+2(r-1)]\in W_{i_{r}}$ for all $\emptyset \neq B\neq
  [i_1+\cdots+i_r+2r] \setminus [i_1+\cdots +i_{r-1}+2(r-1)]$. The identification of $P_1 \times P_2$
  and $P_1' \times P_2'$ provides a copy $P_D'$ of $P_D$ inside $P_1' \times P_2'$ such that
  $P_1' \times \{ \hat{1}\} \subseteq P_D'$.
  Now the restriction $f|_{P_D'}$ of $f$ to $P_D'$ is a closure operator on $P_D'$.
  By $P_1' \times \{ \hat{1}\} \subseteq P_D'$ the projection on the first coordinate
  gives an isomorphism of the image of $f_{P_D}'$ and $P_1'$.
  As in Case 1 we conclude
  \begin{equation*}
    \widetilde{H}_{\ell}(\sd(\Delta)_{W_{i_r}\cup D};\KK)=
       \begin{cases}
          \KK, & \mbox{ if } \ell=i_r\\
            0, & \mbox{otherwise.}
       \end{cases}
  \end{equation*}
  This finishes the proof.
\end{proof}

Applying \ref{thm:continuous} to the sequence $(i_1,\ldots,i_j)=(0,\ldots,0)$ for $1\leq j\leq \frac{d}{2}$ we can deduce the following:

\begin{corollary}\label{le:Betti:bar2}
  Let $1 \leq j \leq \frac{d}{2}$. 
  Then we have
  $\beta_{i,i+j}(\KK[\sd(\Delta_{d-1})]) \neq 0$ for
  $$j \leq i \leq 2^{2j-3}+j.$$
\end{corollary}

The following technical lemmas show that the preceding constructions are sufficient to
to deduce $\beta_{i,i+j}(\KK[\sd(\Delta_{d-1})])\neq 0$ for every $i$ in the intervals we  have identified. 

\begin{lemma}\label{lem:i1=0}
  Let $d\geq 3$, $\frac{d}{2}< j\leq d-1$ be integers and let
  $0 \leq i_1\leq i_2\leq \cdots\leq i_r$  be a sequence of integers such that
  \begin{itemize}
    \item[(i)] $i_1=0$,
    \item[(ii)] $i_1+\cdots +i_r+(r-1)=j-1$,
    \item[(iii)] $i_1+\cdots +i_r+2r\leq d$.
  \end{itemize}
  Then
  \begin{equation*}
    (2^{i_r+2}-2)2^{i_2+\cdots +i_{r-1}+2r-4}+\sum_{\ell=1}^r(2^{i_\ell+2}-2)\geq 2^{j+1}-2.
  \end{equation*}
\end{lemma}

\begin{lemma}\label{lem:i1>1}
   Let $d\geq 3$, $\frac{d}{2}< j\leq d-1$, $r\geq 2$ be integers and let $0 \leq i_1\leq i_2\leq \ldots\leq i_r$ be 
   a sequence of integers such that
   \begin{itemize}
     \item[(i)] $i_1\geq 1$,
     \item[(ii)] $i_1+\cdots +i_r+(r-1)=j-1$,
     \item[(iii)] $i_1+\cdots +i_r+2r\leq d$.
   \end{itemize}
   Set $j_1=i_1-1$, $j_2=i_2+1$ and $j_{\ell}=i_{\ell}$ for $3\leq \ell\leq r$.
   Then
   \begin{equation*}
     (2^{i_r+2}-2)2^{i_2+\cdots +i_{r-1}+2r-4}+\sum_{\ell=1}^r(2^{i_{\ell}+2}-2)\geq
     \sum_{\ell=1}^r(2^{j_{\ell}+2}-2).
   \end{equation*}
\end{lemma}
The proofs of  \ref{lem:i1=0} and \ref{lem:i1>1} are given in \ref{Appendix}.

If we combine \ref{thm:continuous}, \ref{lem:i1=0} and \ref{lem:i1>1} we obtain the following.

\begin{lemma}\label{lem:upperHalf}
  Let $d\geq 1$ and let $\frac{d}{2}< j\leq d-2$.
  Then $\beta_{i,i+j}(\KK[\sd(\Delta_{d-1})])\neq 0$ for all
  \begin{equation*}
    m_j\leq i\leq 2^{j+1}-2-j.
  \end{equation*}
\end{lemma}

\begin{proof}
  Let $0 \leq i_1\leq i_2\leq \ldots\leq i_r$ be a sequence of integers such that
  \begin{itemize}
    \item[(i)] $i_1+\cdots +i_r+(r-1)=j-1$,
    \item[(ii)] $i_1+\cdots +i_r+2r\leq d$, and,
    \item[(iii)] $m_j=\sum_{\ell=1}^r (2^{i_{\ell}+2}-2)-j$.
  \end{itemize}
  If $r=1$, then we have $m_j=2^{j+1}-2-j$ and it suffices to show that
  $\beta_{m_j,m_j+j}(\KK[\sd(\Delta_{d-1})])\neq 0$. But this is true by \ref{le:barysimplex} (ii).

  Assume $r\geq 2$.
  It follows from \ref{thm:continuous} that $\beta_{i,i+j}(\KK[\sd(\Delta_{d-1})])\neq 0$ for all
  \begin{equation}\label{eq:step1}
    m_j\leq i \leq \sum_{\ell=1}^r(2^{i_{\ell}+2}-2)+(2^{i_r+2}-2)2^{i_2+\cdots +i_{r-1}+2r-4}-j.
  \end{equation}
  If $i_1=0$, then \ref{lem:i1=0} directly yields the claim.
  If $i_1\geq 1$, then we infer from \ref{eq:step1} and \ref{lem:i1>1}
  $\beta_{i,i+j}(\KK[\sd(\Delta_{d-1})])\neq 0$ for all
  \begin{equation*}
    m_j\leq i \leq \sum_{\ell=1}^r(2^{j_{\ell}+2}-2)-j,
  \end{equation*}
  where $j_1=i_1-1$, $j_2=i_2+1$ and $j_{\ell}=i_{\ell}$ for $3\leq \ell\leq r$. If we apply \ref{thm:continuous} to the ordered  sequence given by $\{j_1,\ldots,j_r\}$, we can conclude
$\beta_{i,i+j}(\KK[\sd(\Delta_{d-1})])\neq 0$ for all
  \begin{equation*}
    m_j\leq i \leq \sum_{\ell=1}^r(2^{j_{\ell}+2}-2)+(2^{j_r+2}-2)2^{j_2+\cdots +j_{r-1}+2r-4}-j.
  \end{equation*}
If $j_1=0$, the claim follows again from \ref{lem:i1=0}. If $j_1>0$, then we go on perturbing the sequence, i.e., we subtract $1$ from
  its minimum element and add $1$ to its second smallest element until the smallest element equals $0$. Using \ref{thm:continuous} and \ref{lem:i1>1} in each step of this process and finally \ref{lem:i1=0} the claim follows.
\end{proof}

We finally provide the proof of \ref{thm:Betti:bar}.

\begin{proof}[Proof of \ref{thm:Betti:bar}]
 By the same arguments as in the proof of \ref{le:barysimplex} assertions (i) and (ii) are equivalent. We provide the proof of (i).

   Using
   $$\beta_{i,i+j}(\KK[\sd(\Delta_{d-1})]) = \beta_{2^d-d-1-i,2^d-i-j-2}(\KK[\sd(\Delta_{d-1})]$$
   (see \ref{Gorenstein}) it follows from \ref{lem:upperHalf} that $\beta_{i,i+j}(\KK[\sd(\Delta_{d-1})])\neq 0$ for
   \begin{equation*}
      2^d-d-1-(2^{d-1-j+1}-2-(d-1-j))\leq i\leq 2^d-d-1-m_{d-1-j},
   \end{equation*}
   i.e., for
   \begin{equation}\label{eq:BettiFirstPart}
     2^d-2^{d-j}-j\leq i\leq 2^d-d-1-m_{d-1-j}.
   \end{equation}
   In addition, we know from \ref{le:Betti:bar2} and \ref{le:Betti:bar1}
   $\beta_{i,i+j}(\KK[\sd(\Delta_{d-1})])\neq 0$ for
   $j\leq i\leq 2^{2j-3}+j$ and for $2^{j+1}-2 -j \leq i \leq 2^d - 2^{d-j}-1-j$.
   Since $2^d - 2^{d-j}-1-j\geq 2^{2j-3}+j\geq 2^{j+1}-2 -j\geq j$ for $d\geq 3$ and $j\geq 1$,
   we obtain $\beta_{i,i+j}(\KK[\sd(\Delta_{d-1})])\neq 0$ for $j\leq i\leq  2^d - 2^{d-j}-1-j$.
   Combining this with \ref{eq:BettiFirstPart} shows the first part of the claim. The second part, concerning the vanishing of certain Betti numbers, follows from \ref{le:barysimplex} (i) and (ii) combined with \eqref{Gorenstein}.

   (iii) follows from \ref{le:barysimplex} (v).
\end{proof}

\subsection{Asymptotic behavior of Betti numbers for barycentric subdivisions}
  In this section, we do not restrict our attention to barycentric subdivisions of simplices anymore
  but consider iterated barycentric subdivisions of arbitrary simplicial complexes. More precisely,
  let $\Delta$ be a $(d-1)$-simplicial complex and, for $r\in \mathbb{N}$, let $\sd^r(\Delta)$ be
  its $r$\textsuperscript{th} iterated barycentric subdivision,
  defined by $\sd^0(\Delta):=\Delta$, $\sd^1(\Delta):=\sd(\Delta)$ and $\sd^r(\Delta):=\sd(\sd^{r-1}(\Delta))$.
  Given a non-negative integer $0\leq j\leq \reg(\KK[\sd(\Delta)])=\reg(\KK[\sd^r(\Delta)])$
  we are interested in the relative proportion of non-zero Betti numbers
  $\beta_{i,i+j}(\KK[\sd^r(\Delta)])$ compared to the projective dimension
  if $r$ tends to infinity. That is, given $j$, we want to study the quantity
  \begin{equation*}
    \frac{\# \{i~:~\beta_{i,i+j}(\KK[\sd^r(\Delta)])\neq 0\}}{\pdim( \KK[\sd^r(\Delta)])}
  \end{equation*}
  if $r$ goes to infinity.
  Our main result in this section is the following.

\begin{theorem}\label{thm:asymptoticsBary}
  Let $d-1\geq 1$ and let $\Delta$ be a $(d-1)$-dimensional simplicial complex. Let $N(d)$ be the number of vertices of
  $\sd^3(\Delta_{d-1})\setminus\partial(\sd^3(\Delta_{d-1}))$. Then, for $r\geq 3$  
  we have $\beta_{i,i+j}(\KK[\sd^r(\Delta)])\neq 0$
  in the following cases:
  \begin{itemize}
        \item[(i)] $1\leq j\leq \frac{d}{2}$ and
        $j\leq i\leq \pdim (\KK[\sd^r(\Delta)])+\depth (\KK[\Delta])-N(d)+2^d-d-1-m_{d-j-1}$,
    \item[(ii)] $\frac{d}{2}<j\leq d-2$ and $m_j\leq i\leq \pdim (\KK[\sd^r(\Delta)])+
                                                    \depth( \KK[\Delta])-N(d)+2^d-2d+j$, 
     \item[(iii)] $j=d-1$ and $2^d-d-1\leq j\leq \pdim (\KK[\sd^r(\Delta)])+\depth (\KK[\Delta])-N(d)+2^d-d-1$.
  \end{itemize}
\end{theorem}
\begin{proof}
  Since $\Delta$ is a $(d-1)$-dimensional simplicial complex, there exists a $(d-1)$-dimensional face
  $H\in \Delta$. Let $F\in \sd^{r-3}(2^H)$ be a $(d-1)$-dimensional face of $\sd^{r-3}(2^H)$.
  Choose a $(d-1)$-dimensional face $G$ of $\sd^2(2^F )\subseteq \sd^{r-1}(\Delta)$ that
  lies completely in the interior of $\sd^2(2^F )$. After an additional subdivision it
  is guaranteed that none of the vertices of $\sd(2^G )$ is connected by an edge to any of
  the vertices on the boundary of $\sd^3(2^F )$. Moreover, by construction,
  $\sd(2^G )=\sd^r(\Delta)_{\{A~:~\emptyset\neq A\subseteq G \}}$, i.e.,
  $\sd(2^G)$ is an induced subcomplex of $\sd^r(\Delta)$. For simplicity, we use $V_G$ to
  denote the vertex set of $\sd(2^G)$ and $V_F$ to denote the vertices in
  $\sd^3(2^F)\setminus\partial(\sd^3(2^F))$. Observe that we have
  $V_G\subsetneq V_F$. Moreover, let $V_{\Delta}$ be the vertex set of $\sd^r(\Delta)$ and set
  $V:=V_{\Delta}\setminus V_F$. Since in $\sd^3(\Delta)$ there is no edge connecting
  $\sd(2^G)$ and $\partial(\sd^3(2^F))$, there cannot exist an edge passing
  from $V_G$ to $V$. Hence, for any $A\subseteq V$ and $B\subseteq V_G$, it follows that
  $\sd^r(\Delta)_{A\cup B}=\sd^r(\Delta)_A\cup \sd^r(\Delta)_B$ and, using that $V\cap V_G=\emptyset$,
  we infer that $\sd^r(\Delta)_{A\cup B}$ is disconnected with connected components $\sd^r(\Delta)_A$ and
  $\sd^r(\Delta)_B$. This, in particular, implies that
  \begin{equation*}
    \widetilde{H}_{j}(\sd^r(\Delta)_{A\cup B};\KK)=\widetilde{H}_j(\sd^r(\Delta)_A;\KK)\oplus
                                                   \widetilde{H}_j(\sd^r(\Delta)_B;\KK).
  \end{equation*}
  Using that $\widetilde{H}_{j}(\sd^r(\Delta)_B;\KK)=\widetilde{H}_j(\sd(2^G)_B;\KK)$
  we conclude, that if $B\subseteq V_G$ is such that $\widetilde{H}_{j-1}(\sd(2^G)_B;\KK)\neq 0$,
  then $\widetilde{H}_{j-1}(\sd^r(\Delta)_{A\cup B};\KK)\neq 0$ for any $A\subseteq V$. Now Hochster's
  formula \ref{hochster} implies $\beta_{i,i+j}(\KK[\sd^r(\Delta)])\neq 0$ for $l_G(j)\leq i\leq u_G(j)+\# V$,
  where $l_G(j)$ and $u_G(j)$ denote the beginning and the end of the $j$\textsuperscript{th} strand in the
  resolution of $\KK[\sd(\langle G \rangle)]$. If $1\leq j\leq \frac{d}{2}$, it now follows from
  \ref{thm:Betti:bar} (i) that $\beta_{i,i+j}(\KK[\sd^r(\Delta)])\neq 0$ for
  \begin{align*}
     j\leq i & \leq 2^d-d-1-m_{d-j-1}+ \# V\\
             & =\# V_{\Delta}- \# V_F +2^d-d-1-m_{d-j-1}\\
             & =\pdim (\KK[\sd^r(\Delta)])+\depth( \KK[\sd^r(\Delta)])-\# V_F +2^d-d-1-m_{d-j-1}.
  \end{align*}
  Since the depth is invariant under taking barycentric subdivisions \cite[Cor. 2.5]{KW-multiplicity} and since $\# V_F=N(d)$, this shows
  (i).

  The claims in (ii) and (iii) can be seen by the same arguments using parts (ii) and (iii) of
  \ref{thm:Betti:bar}, respectively.  
  \end{proof}
As an immediate consequence of \ref{thm:asymptoticsBary} and the fact that $\lim_{r\rightarrow \infty} \pdim (\KK[\sd^r(\Delta)])=\infty$  we obtain: 
\begin{corollary}\label{coro:asymptoticsBary}
  Let $d-1\geq 1$ and let $\Delta$ be a $(d-1)$-dimensional simplicial complex. For $1\leq j\leq d-1$ one has that $\#\{ i ~:~ \beta_{i,i+j}(\KK[\sd^r(\Delta)])= 0\}$ is bounded above in terms of $j$ and $d$ (and independent of $r$. In particular,
    \begin{equation*}
      \lim_{r\rightarrow \infty}\frac{\#\{i~:~\beta_{i,i+j}(\KK[\sd^r(\Delta)])\neq 0\}}{\pdim (\KK[\sd^r(\Delta)])}
         =1
    \end{equation*}
    for every $j=1,\dots,d-1$. 
\end{corollary}

From the properties of barycentric subdivisions listed in Table \ref{How} we know  that for $r\geq 1$ one has: 
\begin{equation*}
  \reg(\KK[\sd^r(\Delta)])=
  \begin{cases}
    d-1, &\mbox{ if } \widetilde{H}_{d-1}(\Delta;\KK)=0\\
      d, &\mbox{ if } \widetilde{H}_{d-1}(\Delta;\KK)\neq 0.
  \end{cases}
\end{equation*}
In case $\widetilde{H}_{d-1}(\Delta;\KK)=0$, \ref{thm:asymptoticsBary} covers all strands of the minimal
free resolution of $\KK[\sd^r(\Delta)]$.
However, in the second case, \ref{thm:asymptoticsBary} does not provide a
statement for the last strand of the resolution. Indeed, we will see that in this case the situation becomes more involved and
the behavior depends on the geometry of the original simplicial complex $\Delta$. Before we can state
the precise result, we need to introduce some notation and recall some work from \cite{Delucchi}.
For a $(d-1)$-dimensional simplicial complex with $f$-vector $(f_{-1}^\Delta,\ldots, f_{d-1}^\Delta)$
the polynomial
\begin{equation*}
f^{\Delta}(t)=\sum_{j=0}^d f_{j-1}^{\Delta}t^{d-j}
\end{equation*}
is called the \emph{$f$-polynomial} of $\Delta$.
In \cite{BrentiWelker}, Brenti and Welker study the behavior of the $f$-polynomial of $\sd^r(\Delta)$ and show that, as $r\rightarrow \infty$, all but one root of $f^{\sd^r(\Delta)}(t)$ converge to negative real
numbers, depending only on the dimension of $\Delta$, and the last root goes to infinity. This statement was
then made more explicit in \cite{Delucchi}, where ``limit polynomials'' for the normalized $f$-polynomials were
provided. We recall the construction of those polynomials.
Let $\Lambda_d:=(\lambda_{i,j})_{-1\leq i,j\leq d-1}$ be the $(d+1)\times (d+1)$ matrix, where,
$\lambda_{-1,-1}=1$ and $\lambda_{i,-1}=0$ if $i\neq -1$, and $\lambda_{i,j}$ counts the number of
$j$-dimensional faces in the interior of the first barycentric subdivision of an $i$-dimensional simplex,
otherwise. It is shown in  \cite[Lem. 3.4]{Delucchi} that $\Lambda_d$ has eigenvalues $0!,1!,2!,3!,\ldots,
d!$. Let $D_d$ be the diagonal matrix of these eigenvalues (in the stated order) and let $P_d$ be the
corresponding matrix of eigenvectors that diagonalizes $\Lambda_d$, i.e., $D_d=P_d^{-1}\Lambda_d P_d$. We
define $M_{d,d}$ as the $(d+1)\times (d+1)$ matrix, whose only non-zero entry is a $1$ in the lower right
corner. Finally, let $\underline{t}:=(t^d,t^{d-1},\ldots,t^1,t^0)^T$. It is proven in
\cite{Delucchi} that for a $(d-1)$-dimensional simplicial complex $\Delta$ the sequence of normalized
$f$-polynomials $\left(\frac{1}{(d!)^r}f^{\sd^r(\Delta)}(t)\right)_{r\geq 1}$ converges coefficientwise to the polynomial
\begin{equation*}
  p_{\infty}^{\Delta}(t):=(f^{\Delta}P_d)M_{d,d}(P_d)^{-1}\underline{t}.
\end{equation*}
By definition of $P_d$ its last column is eigenvector of the matrix $\Lambda_d$ to the eigenvalue $d!$. Since $\Lambda_d$ is a lower triangular matrix with $\lambda_{d-1,d-1}=d!$, we have $\Lambda_d \fe_{d+1}=d!\fe_{d+1}$ and we can hence choose $\fe_{d+1}$ as last column of $P_d$.
Let $p^{-1}_{d-1,2}$ denote the entry of $P_d^{-1}$ in
the last row and second column. Having set up these additional notations, the above discussion directly
yields the following.

\begin{corollary}\label{cor:vertices}
  Let $\Delta$ be a $(d-1)$-dimensional simplicial complex. Then
  \begin{equation*}
    \lim_{r\rightarrow \infty}\frac{1}{(d!)^r}f_0^{\sd^r(\Delta)}=p^{-1}_{d-1,2}f_{d-1}^{\Delta}.
  \end{equation*}
\end{corollary}

The following simple lemma is needed for the main result of this section and follows immediately from the fact that in a $(d-1)$-dimensional simplicial complex there are no boundaries in
dimension $d-1$.

\begin{lemma}\label{lem:cycles}
Let $\Delta, \Delta'$ be $(d-1)$-dimensional simplicial complexes such that there are geometric realizations for which every $(d-1)$-simplex of $\Delta$ is the union of some $(d-1)$-simplices of $\Delta'$.
Let $\sigma_1,\ldots, \sigma_\ell$ be a basis of the cycle space
  of $\Delta$ in dimension $d-1$ and $\tilde{\sigma}_1,
  \ldots, \tilde{\sigma}_\ell$ their images in the cycle space
  of $\Delta'$. Then every $(d-1)$-cycle of $\Delta'$ is a unique
  linear combination of $\tilde{\sigma}_1,
  \ldots, \tilde{\sigma}_\ell$.
\end{lemma}

The following is a simple consequence of the transformation of $f$-vectors under barycentric subdivision
(see e.g., \cite[Lem. 1]{BrentiWelker}).

\begin{lemma} \label{lem:minimal}
Let $\Delta,\Delta'$ be two $(d-1)$-dimensional simplicial complexes such that for some 
$0 \leq i \leq d-1$ we have $f_i^\Delta > f_i^{\Delta'}$ and $f_{j}^\Delta = f_j^{\Delta'}$ for
$i < j \leq d-1$. Then there exists $R$ such that for $r \geq R$ we have 
  \begin{eqnarray*}
     f_j^{\sd^r(\Delta)}>f_j^{\sd^r(\Delta')} \mbox{~for~} 0 \leq j \leq i \mbox{ and }\\
     f_j^{\sd^r(\Delta)}=f_j^{\sd^r(\Delta')} \mbox{~for~} i < j \leq d-1.
  \end{eqnarray*}
\end{lemma}

The preceding lemma motivates the following minimality concept for $\ell$-cycles of a simplical complex 
$\Delta$. We say that the $\ell$-cycle $\sigma \neq 0$ is \emph{minimal} among the $\ell$-cycles of $\Delta$ 
if there is
no $\ell$-cycle $\sigma' \neq 0$ such that for the simplicial complexes
$\widetilde{\sigma}$ and $\widetilde{\sigma'}$ induced by the support of the cycles there is an index $i$ satisfying $f_i^{\widetilde{\sigma}}>f_i^{\widetilde{\sigma'}}$ and 
$f_j^{\widetilde{\sigma}}=f_j^{\widetilde{\sigma'}}$ for  $i < j \leq d-1$.
We can now formulate our result concerning the last strand of the resolution of $\KK[\sd^r(\Delta)]$,
assuming that $\Delta$ has non-trivial homology in top-dimension. 

\begin{proposition}\label{prop:BarycentricHomology}
 Let $d-1\geq 1$ and let $\Delta$ be a $(d-1)$-dimensional simplicial complex such that  and 
  $\widetilde{H}_{d-1}(\Delta;\KK)\neq 0$.
  Let further $\sigma$ be a minimal homology $(d-1)$-cycle of $\Delta$ and let
  \begin{equation*}
    \widetilde{\sigma}=\{F\in \Delta~:~F\subseteq G\mbox{ for some } G \mbox{ in the support of } \sigma \}
  \end{equation*}
 be the corresponding induced subcomplex of $\Delta$.
  Then
  \begin{itemize}
    \item[(i)] for $r\geq 1$ $\beta_{i,i+d}(\KK[\sd^r(\Delta)])$ for $\# V_r^{\sigma}  -d\leq i\leq \pdim( \KK[\sd^r(\Delta)])$, where $V_r^{\sigma}$ denotes the vertex set of $\widetilde{\sigma}^{\langle r\rangle}$. If $r$ is large, then in addition  $\beta_{i,i+d}(\KK[\sd^r(\Delta)])=0$ for $0\leq i<\# V_r^{\sigma}  -d$.
    \item[(ii)]
      \begin{equation*}
        \lim_{r\rightarrow\infty}\frac{\#\{i~:~
              \beta_{i,i+d}(\KK[\sd^r(\Delta)])\neq 0\}}{\pdim (\KK[\sd^r(\Delta)])}=1-
                      \frac{f_{d-1}^{\widetilde{\sigma}}}{f_{d-1}^{\Delta}}.
      \end{equation*}
  \end{itemize}
\end{proposition}

\begin{proof}
  One observes that $\sd^r(\Delta)_{V_r^{\sigma}}=\sd^r(\widetilde{\sigma})$. Since $\sigma$ is a homology $(d-1)$-cycle of $\Delta$, the $r$\textsuperscript{th} barycentric subdivision $\sd^r(\widetilde{\sigma})$ of its induced complex $\widetilde{\sigma}$ 
  gives rise to a homology $(d-1)$-cycle $\sigma_r$ of $\sd^r(\Delta)$ and  
  we conclude that $\widetilde{H}_{d-1}(\sd^r(\Delta)_{V_r^{\sigma}};\KK)\neq 0$. It
  follows from Hochster's formula \ref{hochster} that for
  $i=\# V_r^{\sigma}-d$ we have $\beta_{i,i+d}(\KK[\sd^r(\Delta)])\neq 0$. 
  Moreover, if we consider induced
  subcomplexes $\sd^r(\Delta)_A$ of $\sd^r(\Delta)$ with $V_r^{\sigma}\subseteq A$, then $\sigma_r$ will remain
  a homology $(d-1)$-cycle in $\sd^r(\Delta)_A$. This in particular implies
  $\beta_{i,i+d}(\KK[\sd^r(\Delta)])\neq 0$ for $\# V_r^{\sigma}-d\leq i\leq \pdim (\KK[\sd^r(\Delta)])$
  and hence the part of (i) concerning the non-vanishing Betti numbers  follows.
  For the vanishing part, let $\tau$ be a $(d-1)$-cycle of
  $\sd^r(\Delta)$. By \ref{lem:cycles} it follows that $\tilde{\tau}$ is the union of some 
  $\sd^r(\tilde{\sigma'})$ for $(d-1)$-cycles $\sigma'$ of $\Delta$. Thus for $r$ large enough by
  \ref{lem:minimal} it follows that the vertex set of $\tilde{\tau}$ is of larger cardinality than 
  the vertex set of $\sd^r(\tilde{\sigma})$ for the minimal $(d-1)$-cycle $\sigma$ of $\Delta$.
  Now Hochster's formula \ref{hochster} shows the vanishing.

  It remains to show (ii). Let $V_r^{\Delta}$ be the vertex set of $\sd^r(\Delta)$. We know from (i) that
  \begin{align*}
    &\empty\frac{1}{\pdim (\KK[\sd^r(\Delta)])} \#\{i~:~\beta_{i,i+d}(\KK[\sd^r(\Delta)])\neq 0\}\\
     &= \frac{1}{\pdim( \KK[\sd^r(\Delta)])}\left(\pdim (\KK[\sd^r(\Delta)])-(\# V_r^{\sigma}-d-1)\right)\\
     &=1 - \frac{\# V_r^{\sigma}-d-1}{\# V_r^{\Delta}-\depth (\KK[\sd^r(\Delta)])}\\
     &= 1 - \frac{\frac{1}{(d!)^r}(\# V_r^{\sigma}-d-1)}{\frac{1}{(d!)^r}(\# V_r^{\Delta}-\depth( \KK[\Delta]))}.
  \end{align*}
  We infer from \ref{cor:vertices} that numerator and denominator of this fraction converge to $p^{-1}_{d-1,2}f_{d-1}^{\widetilde{\sigma}}$ and
  $p^{-1}_{d-1,2}f_{d-1}^{\Delta}$, respectively. Thus, we can finally
  conclude,
  \begin{equation*}
    \lim_{r\rightarrow \infty}\frac{1}{\pdim (\KK[\sd^r(\Delta)]})
         \# \{i~:~\beta_{i,i+d}(\KK[\sd^r(\Delta)])\neq 0\} =
         1-\frac{p^{-1}_{d-1,2}f_{d-1}^{\widetilde{\sigma}}}{p^{-1}_{d-1,2}f_{d-1}^{\Delta}p_{d+1}}=
         1-\frac{f_{d-1}^{\widetilde{\sigma}}}{f_{d-1}^{\Delta}}.
  \end{equation*}
  \end{proof}

  We now provide an example that shows that for any $d$ indeed any rational number in the half-open intervall $[0,1)$ can occur as limit in (ii) of the above proposition.
  \begin{example}\label{ex:limit}
Let $\frac{p}{q}\in [0,1)\cap\mathbb{Q}$. \\
{\sf Case 1:}  $\frac{p}{q}=0$.\\
Let $\Delta$ be the boundary of a $(d-1)$-simplex. In this case, the only minimal $(d-1)$-homology cycle of $\Delta$ is $\Delta$ itself, and it follows from Proposition \ref{prop:BarycentricHomology} that 
\begin{equation*}
\lim_{r\rightarrow \infty}\frac{\# \{i~:~\beta_{i,i+d}(\KK[\sd^r(\Delta)])\neq 0\}}{\pdim (\KK[\sd^r(\Delta)])}=0.
\end{equation*}
Indeed, for any $r$, we have $\beta_{i,i+d}(\KK[\sd^r(\Delta)])\neq 0$ if and only if $i=\pdim (\KK[\sd^r(\Delta)])$.\\

\noindent{\sf Case 2:} $p>0$.\\
Our construction relies on a result of Lee \cite{Lee} and 
Bj\"orner and Linusson \cite[Theorem 1, Theorem 7]{BjoernerLinusson}. 
They showed that for any $d$, there exists $N(d)\in \mathbb{N}$ such that 
for all even numbers $n\geq N(d)$ there exists a simple $d$-polytope with 
$n$ vertices or, by taking the dual, a simplicial $d$-polytope with $n$ 
facets, i.e., a simplicial polytopal $(d-1)$-sphere with $n$ facets.

By this result there exists a simplicial $(d-1)$-sphere $\Delta_1$ with $2N(d)(q-p)$ 
facets. Let $\Delta$ be the Cohen-Macaulay complex obtained from $\Delta_1$ by stacking $2N(d)p$ copies of the $(d-1)$-simplex over a specified $(d-2)$-face of $\Delta_1$. In this case $\Delta_1$ is the only minimal 
$(d-1)$-homology cycle of $\Delta$ and \ref{prop:BarycentricHomology} 
yields 
\begin{equation*}
        \lim_{r\rightarrow\infty}\frac{\#\{i~:~
              \beta_{i,i+d}(\KK[\sd^r(\Delta)])\neq 0\}}{\pdim (\KK[\sd^r(\Delta)])}=1-\frac{2N(d)(q-p)}{2N(d)q}=1-\frac{q-p}{q}=\frac{p}{q}.
\end{equation*}
  \end{example}

\section{Edgewise subdivisions}
\subsection{Algebraic invariants}

We have listed in Table \ref{How} how  the  basic invariants behave under the $r$\textsuperscript{th} edgewise subdivision of a $(d-1)$-dimensional simplicial complex $\Delta$. It remains to provide a precise statement concerning the largest degree $t_1(\Delta^{\langle r\rangle})$ of a minimal generator of $I_{\Delta^{\langle r\rangle}}$.

\begin{lemma}\label{lem:minGen}
 Let $\Delta$ be a $(d-1)$-dimensional simplicial complex. Let $N(\Delta)$ be the set of minimal non-faces of $\Delta$ of cardinality $t_1(\KK[\Delta])$. Then:
\begin{itemize}
\item[(i)] If $\Delta$ is flag or $\Delta=\Delta_{d-1}$, then $t_1(\Delta^{\langle r\rangle})=2$.
\item[(ii)] If $\Delta$ is neither flag nor a $(d-1)$-simplex, the following two cases can occur. 
 If there exists $F\in N(\Delta)$ and a vertex $v\in \Delta$ such that $\partial(F)\ast \{v\} \subseteq \Delta$, then, for any $r\geq 2$,  $t_1(\KK[\Delta^{\langle r\rangle}])=t_1(\KK[\Delta])$. Otherwise, $t_1(\KK[\Delta^{\langle r\rangle}])=t_1(\KK[\Delta])-1$.
\end{itemize}
\end{lemma}

\begin{proof}
Since flag-ness is preserved under edgewise subdivisions, we have $t_1(\Delta^{\langle r\rangle})=2$ if $\Delta$ is flag. If $\Delta$ is a $(d-1)$-simplex, then $\Delta^{\langle r\rangle}$ is flag by the definition of edgewise subdivision, which shows $t_1(\Delta^{\langle r\rangle})=2$. 

To show (ii), first we prove that $t_1(\Delta^{\langle r\rangle})\leq t_1(\Delta)$. For this aim, let $G=\{v_1,\ldots,v_m\}$ be a minimal non-face of $\Delta^{\langle r\rangle}$. Then at least one of the two conditions in the definition of edgewise subdivisions fails. If the second one fails, we have $|G|=2$ and hence $G$ gives rise to a minimal generator of $I_{\Delta^{\langle r\rangle}}$ of degree $2<t_1(\Delta)$. So assume, only the first condition fails, i.e., $H:=\bigcup_{a\in G}\supp(a)\notin \Delta$. There exists $F\in N(\Delta)$ such that $F\subseteq H$. Since $G$ is minimal, for any $1\leq j\leq m$, there exists $i_j\in F$ such that $i_j\in \supp(v_j)$ and $i_j\notin \supp(v_\ell)$ for $1\leq \ell \leq m$ and $\ell \neq j$. This implies that $\#G=m\leq s$, which finally shows $t_1(\Delta^{\langle r\rangle})\leq t_1(\Delta)$.\\
Next, we show $t_1(\Delta^{\langle r\rangle})\geq t_1(\Delta)-1$. Let $F\in N(\Delta)$. Without loss of generality let $F=[s]$, where $s=t_1(\Delta)$. For $1\leq i\leq s-1$ set $v_i=\fe_i+(r-1)\fe_s$ and let $G=\{v_1,\ldots,v_{s-1}\}$. Then, $H\in \Delta^{\langle r\rangle}$ for any $H\subsetneq G$ but $G\notin \Delta^{\langle r\rangle}$, i.e., $G$ is a minimal non-face of $\Delta^{\langle r\rangle}$. This implies the claimed inequality. \\
Now assume that there exists $F\in N(\Delta)$ and a vertex $v\in \Delta$ such that $\partial(F)\ast \{v\} \subseteq \Delta$. Without loss of generality assume $F=[s]$ and $v=s+1$. We set $w_i=\fe_i+(r-1)\fe_{s+1}$ for $1\leq i\leq s$ and $G=\{w_1,\ldots,w_s\}$. By minimality of $F$ we have $H\in \Delta^{\langle r\rangle}$ for any $H\subsetneq G$. Since $F$ is a non-face, we further know that $G\notin \Delta^{\langle r\rangle}$. Hence, $G$ is a minimal non-face of $\Delta^{\langle r\rangle}$ of cardinality $t_1(\Delta)$, which implies $t_1(\Delta^{\langle r\rangle})\geq t_1(\Delta)$ and hence $t_1(\Delta^{\langle r\rangle})=t_1(\Delta)$ in this case.\\
Finally assume that for all $F\in N(\Delta)$ and all vertices $v\in \Delta$ we have $\partial(F)\ast \{v\} \not\subseteq \Delta$. Let $G=\{v_1,\ldots,v_m\}$ be a minimal non-face of $\Delta^{\langle r\rangle}$ such that $F\bigcup_{i=1}^m \supp(v_i)$, where $F$ is a minimal non-face of $\Delta$. Without loss of generality assume $F=[s]$. It follows from the second paragraph of this proof that $\#G\leq \#F$. If $\#G<\#F$, then $G$ corresponds to a minimal generator of $I_{\Delta^{\langle r\rangle}}$ of degree $<t_1(\Delta)$. So assume $\#G=\#F$. From the second paragraph of this proof it follows that for any $1\leq j\leq s$ there exists a unique $1\leq i_j\leq s=m$, such that $j\in\supp(v_{i_j})$. Since $\{v_\ell,v_k\}\in \Delta^{\langle r\rangle}$ for any $1\leq \ell<k\leq s$, it must hold $(v_{i_j})_j =1$ for $1\leq j\leq s$. Hence, any vertex $v_j\in G$ is of the form $v_j=\fe_{i_j}+z_j$, where $z_j\in \Delta^{\langle r-1\rangle}$ and $\supp(z_i)\cap [s]=\emptyset$. Since $G
 $ is a minimal non-face of $\Delta^{\langle r\rangle}$, the boundary of $G$ is a subcomplex of $\Delta^{\langle r\rangle}$ and hence $\partial (F)\ast\bigcup_{i=1}^s\supp(z_i)$ is a subcomplex of $\Delta$. Since  $\bigcup_{i=1}^s\supp(z_i)\neq \emptyset$ ($r\geq 2$), we arrive at a contradiction and the claim follows. 
\end{proof}

In the following we will analyze the Betti numbers of the $r$\textsuperscript{th} edgewise subdivision of the $(d-1)$-dimensional simplex $\Delta_{d-1}$ . We start by  showing  that, locally, $\Delta_{d-1}^{\langle r \rangle}$  behaves as a barycentric subdivision of the boundary of a $(d-1)$-simplex. 

\begin{proposition}\label{lem:InnerFace}
  Let $r\geq d$ be positive integers. Then for any face $F$ of $\Delta_{d-1}^{\langle r \rangle}$,
  satisfying
  \begin{eqnarray} \label{eq:face}
     \partial|F| & = & |F| \cap \partial |\Delta_{d-1}^{\langle r \rangle}|,
  \end{eqnarray}
  the link
  $\lk_{\Delta_{d-1}^{\langle r \rangle}}(F)$ is abstractly isomorphic
  to the barycentric subdivision of the boundary of a $(d-\# F)$-simplex.
\end{proposition}

\begin{proof}
  In order to prove the statement we will first list two facts that allow
  to simplify the situation.

  \begin{itemize}
    \item If $F$ and $G$ are two faces of equal dimension in
      $\Delta_{d-1}^{\langle r \rangle}$ satisfying \ref{eq:face},
      then it is straightforward to show that their links are isomorphic as
      simplicial complexes.
    \item Let $F_1$ be a face of
      $\Delta_{d-1}^{\langle r \rangle}$ and $F_2$ a face of
      $\Delta_{d-1}^{\langle s\rangle}$ both satisfying \ref{eq:face} in their respective complexes.
      If
      $\# F_1= \# F_2=t$ and $r,s\geq d-t+1$, then it is easy to show that
      $\lk_{\Delta_{d-1}^{\langle r \rangle}}(F_1)$ and
      $\lk_{\Delta_{d-1}^{\langle s\rangle}}(F_2)$ are also isomorphic as simplicial
      complexes.
  \end{itemize}

  Combining those two reductions, we will now show the claim
  for a specific face $F$ in the $(d-\# F +1)$\textsuperscript{st} edgewise
  subdivision of $\Delta_{d-1}$.
  More precisely, let $s$ be a fixed positive integer. Let
  \begin{align*}
     w_i&=(\underbrace{1,\ldots,1}_{d-s},\underbrace{0,\ldots,0}_{i- 1},
          \underbrace{1}_{\mbox{position }d+i-s},\underbrace{0,\ldots,0}_{s-i})\\
        &=\sum_{j=1}^{d-s}\fe_j+\fe_{d-s+i},
  \end{align*}
  \noindent for $1\leq i\leq s$. Let
  $F$ be the abstract simplex whose geometric realization has vertex set $\{w_1,\ldots,w_s\}$.
  Then all vertices of $F$ lie on the boundary and so does the convex hull of any $s-1$ subset of the
  vertices. But $F$ itself does not lie on the boundary of the simplex. Hence $F$ satisfies
  \ref{eq:face}. It now suffices to show that
$\lk_{\Delta_{d-1}^{\langle d-s+1\rangle}}(F)$ is isomorphic to the barycentric
subdivision of the boundary of a $(d-s)$-simplex. Let $V$ be the vertex
set of $\lk_{\Delta_{d-1}^{\langle d-s+1\rangle}}(F)$. We define a map
$\Phi: V \rightarrow \{A~:~\emptyset \neq A\subsetneq [d-s]\cup \{d\}\}$ by
\begin{align*}
       v &\mapsto \left\{
           \begin{array}{ll} \supp(\ffi(v-w_1))\cap [d-s], & \mbox{ if } \ffi(v-w_1)\in \{0,1\}^d\\ & \\
                     \left[d-s\right]\setminus  \supp(\ffi(v-w_1))\cup\{d\} , & \mbox{ if } \ffi(v-w_1)\in \{-1,0\}^d.
                        \end{array} \right.
 \end{align*}

We claim that $\Phi$ is a bijection and moreover, that it induces an
isomorphism between $\lk_{\Delta_{d-1}^{\langle d-s+1\rangle}}(F)$ and
$\sd(\partial\Delta_{d-s})$ as simplicial complexes, where for the purpose of this proof
$\Delta_{d-s}$ denotes the $(d-s)$-simplex on vertex set
$[d-s]\cup \{d\}$.

{\sf Injectivity:} Let $u$, $v\in \lk_{\Delta_{d-1}^{\langle d-s+1\rangle}}(F)$ and
$u\neq v$. If $\ffi(u-w_1)\in \{0,1\}^d$ and
$\ffi(v-w_1)\in \{-1,0\}^d$ (or vice versa), then $d\in \Phi(v)$ but
$d\notin \Phi(u)$. Hence, $\Phi(u)\neq \Phi(v)$ in this case. Now, let
$\ffi(u-w_1)\in \{0,1\}^d$ and $\ffi(v-w_1)\in \{0,1\}^d$. Assume, by
contradiction, that $\Phi(u)=\Phi(v)$, i.e.,
$\supp(\ffi(u-w_1))\cap [d-s]=\supp(\ffi(v-w_1))\cap [d-s]$ and hence
the first $d-s$ components of $\ffi(u-w_1)$ and $\ffi(v-w_1)$ are
equal. This implies that the first $d-s$ components of $u$ and $v$
coincide. If $\ffi(u-w_1)_{d-s}=\ffi(v-w_1)_{d-s}=1$, it must hold
that $u_{d-s+1}=\ldots=u_d=0$ and $v_{d-s+1}=\ldots =v_d=0$ since
$\sum_{j=1}^du_j=\sum_{j=1}^d v_j=d-s+1$. This, in particular implies
that $u=v$. If $\ffi(u-w_1)_{d-s}=\ffi(v-w_1)_{d-s}=0$, we can infer
from $(w_1)_{d-s+1}=1$ and $\ffi(u-w_1)\in\{0,1\}^d$, that
$u_{d-s+1}=v_{d-s+1}=1$. As in the previous case, we conclude that
$u_{d-s+2}=\ldots =u_d=0$ and $v_{d-s+2}=\ldots =v_d=0$. Hence, again
$u=v$. Finally let $\ffi(u-w_1)\in \{-1,0\}^d$ and
$\ffi(v-w_1)\in \{-1,0\}^d$. Similar arguments as in the previous case
show that we must have $\Phi(u)\neq \Phi(v)$.

{\sf Surjectivity:} Let $\emptyset\neq G\subsetneq [d-s]\cup\{d\}$.
First assume that $d\notin G$. We define a vector $v\in \mathbb{Z}^d$
by setting $v_1=1$ if $1\notin G$ and $v_1=2$ if $1\in G$ and
successively, for $2\leq j\leq d-s$
\begin{align*}
    v_j=
        \begin{cases}
    -\sum_{\ell=1}^{j-1}v_\ell+j, &\mbox{ if } j\notin G\\
    -\sum_{\ell=1}^{j-1}v_\ell+j+1, &\mbox{ if } j\in G. \\
        \end{cases}
\end{align*}

Moreover, $v_{d-s+1}=0$ if $d-s\in G$ and $v_{d-s+1}=1$ if
$d-s\notin G$. For $d-s+2\leq j\leq d$, we set $v_j=0$. It is
straightforward to verify that $v\in \lk_{\Delta_{d-1}^{\langle d-s+1\rangle}}(F)$ and
$\Phi(v)=G$.
Now, suppose $d\in G$. In this case, we define a vector $v\in \mathbb{Z}^d$ by setting $v_1=1$ if $1\in G$ and $v_1=0$ if $1\notin G$. For $2\leq j\leq d-s$ we successively set
\begin{align*}
    v_j=
        \begin{cases}
    -\sum_{\ell=1}^{j-1}v_\ell+j, &\mbox{ if } j\notin G\\
    -\sum_{\ell=1}^{j-1}v_\ell+j-1, &\mbox{ if } j\in G. \\
        \end{cases}
\end{align*}
For $d-s+2\leq j\leq d-1$ we set $v_j=0$. In addition, we let $v_{d
-s+1}=1$ and $v_d=0$ if $d-s\notin G$ and $v_{d-s+1}=0$ and $v_d=1$ if
$d-s\in G$. One can easily check that indeed $v\in \lk_{\Delta_{d
-1}^{\langle d-s+1\rangle}}(F)$ and $\Phi(v)=G$, which completes the proof of
surjectivity of $\Phi$.

We can extend the map $\Phi$ to $\lk_{\Delta_{d-1}^{\langle d-s+1\rangle})}(F)$ by
mapping a face $A=\{a_1,\ldots,a_t\}$ of $\lk_{\Delta_{d-1}^{\langle d-s+1\rangle}}(F)$
to $\{\Phi(a_1),\ldots,\Phi(a_t)\}$. We need to show that
$\lk_{\Delta_{d-1}^{\langle d-s+1\rangle}}(F)$ and  $\sd(\partial\Delta_{d-s})$ are
isomorphic as simplicial complexes. Since both complexes are flag, it
suffices to show that $\Phi$ induces an isomorphism between their
$1$-skeleta. Let $\{v,w\}\in \lk_{\Delta_{d-1}^{\langle d-s+1\rangle}}(F)$. Then, we
either have $\ffi(v-w)\in\{0,1\}^d$ or $\ffi(v-w)\in\{-1,0\}^d$.
Without loss of generality we can assume that $\ffi(w-v)\in\{0,1\}^d$.
Moreover, since $v\in \lk_{\Delta_{d-1}^{\langle d-s+1\rangle}}(F)$, we know that
$\ffi(v-w_1)\in\{0,1\}^d$ or $\ffi(v-w_1)\in\{-1,0\}^d$.

\noindent{\sf Case 1:} $\ffi(v-w_1)\in\{0,1\}^d$\\
Let $1\leq \ell\leq d-s$ such that $\ffi(v-w_1)_{\ell}=1$. It follows
that
\begin{align*}
    \ffi(w-w_1)_\ell&=\ffi(w-v)_\ell+\ffi(v-w_1)_\ell\\
    &\geq \ffi(v-w_1)_\ell=1,
\end{align*}
\noindent since $\ffi(w-v)\in \{0,1\}^d$. Since $w\in \lk_{\Delta_{d-1}^{\langle d-s+1\rangle}}(F)$, we conclude $\ffi(w-w_1)_\ell=1$ and
$\ffi(w-w_1)\in \{0,1\}^d$. We can deduce $\supp(\ffi(v-
w_1))\subsetneq \supp(\ffi(w-w_1))$ and hence, $\Phi(v)\subsetneq
\Phi(w)$, i.e., $\{\Phi(v),\Phi(w)\}\in \sd(\partial\Delta_{d-s})$.

\noindent{\sf  Case 2:} $\ffi(v-w_1)\in\{-1,0\}^d$\\
We consider the two subcases, that $\ffi(w-w_1)\in\{-1,0\}^d$ and
$\ffi(w-w_1)\in\{0,1\}^d$. Suppose that we are in the first case. If
we have $\ffi(w-w_1)_\ell=-1$ for some $1\leq \ell\leq d-s$, then it follows from 
$\ffi(v-w)\in\{-1,0\}^d$ that
\begin{align*}
    \ffi(v-w_1)_\ell&=\ffi(v-w)_\ell+\ffi(w-w_1)_\ell\\
    &\leq \ffi(w-w_1)_\ell=-1.
\end{align*}
This implies $\supp(\ffi(w-w_1))\subsetneq\supp(\ffi(v-w_1))$ and hence
\begin{equation*}
\Phi(v)=[d-s]\setminus \supp(\ffi(v-w_1))\cup\{d\}\subsetneq
[d-s]\setminus \supp(\ffi(w-w_1))\cup\{d\}=\Phi(w).
\end{equation*}
In particular, $\{\Phi(v),\Phi(w)\}$ is an edge of
$\sd(\partial(\Delta_{d-s}))$
It remains to consider the case $\ffi(w-w_1)\in\{0,1\}^d$. Let $1\leq
\ell \leq d-s$ with $\ffi(w-w_1)_\ell=1$. We have
\begin{align*}
    \ffi(v-w_1)_\ell&=\ffi(v-w)_\ell+\ffi(w-w_1)_\ell\\
    &=\ffi(v-w)_\ell+1.
\end{align*}
Since $\ffi(v-w)\in\{-1,0\}^d$ and $\ffi(v-w_1)\in \{-1,0\}^d$, it
must hold that $\ffi(v-w)_\ell=-1$ and $\ffi(v-w_1)_\ell=0$. Using
that $d\in \Phi(v)$ but $d\notin \Phi(w)$, we can finally conclude that
\begin{equation*}
    \Phi(w)=\supp(\ffi(w-w_1))\cap [d-s]\subsetneq [d-s]\setminus\supp(\ffi(v-w_1))\cup\{d\}=\Phi(v)
\end{equation*}
and thus $\{\Phi(v),\Phi(w)\}\in \sd(\partial\Delta_{d-s})$.
This finishes the proof of containment of the $1$-skeleton of
$\lk_{\Delta_{d-1}^{\langle d-s+1\rangle}}(F)$ in the $1$-skeleton of $\sd(\partial
\Delta_{d-s})$. We omit the proof of the other inclusion since the
it follows from a similar reasoning.
\end{proof}

The following corollary is a special case of \ref{lem:InnerFace}.

\begin{corollary}\label{lem:InnerVertex}
  Let $r\geq d$ be positive integers. Then for any vertex $v = (i_1,\ldots,i_n)$ in the interior of
  $\Delta_{d-1}^{\langle r \rangle}$ the link $\lk_{\Delta_{d-1}^{\langle r \rangle}}(v)$ is abstractly isomorphic to the barycentric subdivision of the boundary of a $(d-1)$-simplex.
\end{corollary}

Note that necessarily $r\geq d$ if there is an interior vertex of
$\Delta_{d-1}^{\langle r\rangle}$. This will be also crucial for \ref{cor:reg}.
Using \ref{lem:InnerVertex} and \ref{thm:Betti:bar} we get the following bounds for the non-vanishing of Betti numbers in
the edgewise subdivision. We use the constant $m_j$ defined by \ref{eq:mj}.

\begin{corollary}
\label{Prop:edgewise1}
  Let $r \geq d$. Then
  \begin{itemize}
    \item[(i)] For $1 \leq j \leq \frac{d}{2}$
          we have $\beta_{i,i+j}(\KK[\Delta_{d-1}^{\langle r \rangle}]) \neq 0$ for
          $j \leq i \leq 2^d-d-1-m_{d-1-j}$. 
              \item[(ii)] For $\frac{d}{2} < j \leq d-2$ we have
          $\beta_{i,i+j}(\KK[\Delta_{d-1}^{\langle r \rangle}]) \neq 0$ for
          $m_j \leq i \leq 2^d-2d+j$.
          \item[(iii)] For $j=d-1$, we have $\beta_{2^d-1-d,2^d-1-d+j}(\KK[\Delta_{d-1}^{\langle r\rangle}])\neq 0$.
  \end{itemize}
\end{corollary}

The next corollary shows that  \ref{Prop:edgewise1} covers all but possibly the last strand in the minimal free resolution of $\KK[\Delta_{d-1}^{\langle r\rangle}]$.

\begin{corollary}\label{cor:reg}
 Let $\Delta$ be a $(d-1)$-dimensional simplicial complex. Then
 $$ \reg (\KK[\Delta^{\langle r \rangle}])=
\begin{cases}
 d-1, &\mbox{ if } \widetilde{H}_{d-1}(\Delta;\KK)=0 \mbox{ and } r\geq d\\
d, &\mbox{ if } \widetilde{H}_{d-1}(\Delta;\KK)\neq 0.
\end{cases}
$$
Moreover, for $1\leq r\leq d-1$ one has $ \reg (\KK[\Delta^{\langle r \rangle}])\geq \max(\reg(\KK[\Delta]),r-1)$.
\end{corollary}

\begin{proof}
Let $r\geq d$ and let $v$ be an interior vertex of the $r$\textsuperscript{th} edgewise subdivision of a $(d-1)$-simplex $F\in \Delta$. Then $\lk_{\Delta^{\langle r \rangle}}(v)=\lk_{(2^F)^{\langle r \rangle}}(v)$ and it follows from \ref{lem:InnerVertex}, that $\lk_{\Delta^{\langle r \rangle}}(v)$ is abstractly isomorphic to the barycentric subdivision of the boundary of a $(d-1)$-simplex. Let $V$ be the vertex set of $\lk_{\Delta^{\langle r \rangle}}(v)$. Then $\lk_{\Delta^{\langle r \rangle}}(v)=(\Delta^{\langle r\rangle})_{V}$. Using Hochster's formula \ref{hochster} we conclude that $\reg(\KK[\Delta^{\langle r \rangle}])\geq d-1$. If $\widetilde{H}_{d-1}(\Delta;\KK)=0$, then  we also have $\widetilde{H}_{d-1}(\Delta^{\langle r\rangle};\KK)=0$ and moreover, $\widetilde{H}_{d-1}((\Delta^{\langle r\rangle})_W;\KK)=0$ for any subset $W$ of the vertices of $\Delta^{\langle r\rangle}$. Hence, Hochster's formula \ref{hochster} implies $\reg(\KK[\Delta^{\langle r \rangle}])\leq 
 d-1$, which shows the first part.

If $\widetilde{H}_{d-1}(\Delta;\KK)\neq 0$, then we also have $\widetilde{H}_{d-1}(\Delta^{\langle r\rangle};\KK)\neq 0$ and by an application of Hochster's formula \ref{hochster} we infer  $\reg(\KK[\Delta^{\langle r \rangle}])\geq d$. On the other hand, the regularity of $\KK[\Gamma]$ of any $(d-1)$-dimensional simplicial complex $\Gamma$ cannot exceed $d$ and the claim follows.

For the last part, we first show $\reg (\KK[\Delta^{\langle r \rangle}])\geq \reg(\KK[\Delta])$. Let $\reg(\KK[\Delta])=s$. By Hochster's formula \ref{hochster} there exists a subset $W$ of the vertex set of $\Delta$ such that $\widetilde{H}_{s-1}(\Delta_W;\KK)\neq 0$. This implies $\widetilde{H}_{s-1}((\Delta_W)^{\langle r\rangle};\KK)\neq 0$. Let $W_r$ be the vertex set of $(\Delta_W)^{\langle r\rangle}$. Since $(\Delta^{\langle r\rangle})_{W_r}$ is a deformation retract of $(\Delta_W)^{\langle r\rangle}$, it follows that $\widetilde{H}_{s-1}((\Delta^{\langle r\rangle})_{W_r};\KK)\neq 0$ and by Hochster's formula \ref{hochster} we conclude $\reg(\KK[\Delta^{\langle r\rangle}])\geq \reg (\KK[\Delta])$. To see the other inequality, consider an $(r-1)$-dimensional face $F$ of $\Delta$. Let $V_F$ be the vertex set of $(\partial(2^F))^{\langle r\rangle}$. Since $(2^F)^{\langle r\rangle}$ has an interior vertex, we infer $(\Delta^{\langle r\rangle})_{V_F}=(\partial(F))^{\langle r
 \rangle}$ and hence $\widetilde{H}_{r-2}((\Delta^{\langle r\rangle})_{V_F};\KK)\neq 0$. The claim follows from Hochster's formula \ref{hochster}.
\end{proof}

The above corollary in particular shows that the regularity can only increase under arbitrary edgewise subdivision.

The upper bounds for the non-vanishing of the Betti numbers in \ref{Prop:edgewise1} can be improved further.
Indeed, it can be shown that the strands in the Betti diagram of the $r$\textsuperscript{th} edgewise
subdivision of $\Delta_{d-1}^{\langle r \rangle}$ run up to the projective dimension. To provide these results we need to better understand the topology of edgewise subdivisions.

\begin{lemma}
  \label{lem:homotopy}
  Let $\Delta$ be a simplicial complex on ground set $\Omega$ such that 
  $|\Delta|$ is a regular triangulation
  of a $(d-1)$-ball and let $F \in \Delta$.
  If $\partial |F| = |F| \cap \partial |\Delta|$, then there are deformation retractions from $|\Delta|\setminus |F|$ to $|\Delta_{\Omega\setminus F}|$ to $|\lk_{\Delta}(F)|$.
\end{lemma}

\begin{proof}
  By \cite[Lem. 4.7.27]{BLSWZ} or \cite[Lem. 70.1]{Munkres} $|\Delta_{\Omega\setminus F}|$ is a deformation
  retract of $|\Delta|\setminus |F|$.
  Since $|\Delta|$ is a regular triangulation, we know that $|\st_{\Delta}(F)|$ is convex. Let $p$ be an
  interior point of $|F|$. Consider the map that sends a point $q\in |\Delta_{\Omega\setminus F}|$ to
  the intersection of $\partial|\st_{\Delta}(F)|$ and the line segment through $p$ and $q$. Note that this
  intersection is well-defined since $|\st_{\Delta}(F)|$ is convex. The image of this map is
  $|\st_{\Delta}(F)|\setminus |F|$. Thus, the map defines a deformation retract between
  $|\Delta_{\Omega \setminus F}|$ and $\partial|\st_{\Delta}(F)|\setminus |F|$. Let $\Gamma$ be the
  simplicial complex whose geometric realization is $\partial|\st_{\Delta}(F)|$. Another application
  of \cite[Lem. 4.7.27]{BLSWZ} or \cite[Lem. 70.1]{Munkres} shows that
  $|\Gamma_{\Omega\setminus F}|=|\lk_{\Delta}(F)|$ is a deformation retract of $|\Gamma|\setminus |F|$.
\end{proof}

Note that by definition $\partial |F|=\emptyset$ if $F$ is  a $0$-dimensional face.

\begin{lemma}
  \label{lem:nonvanishing}
  Let $\Delta$ be a simplicial complex on ground set $\Omega$ such that $|\Delta|$ is a regular triangulation
  of a $(d-1)$-ball and let $F \in \Delta$ such that $\partial|F|\subseteq \partial|\Delta|$. Let
  $B\subseteq \Omega\setminus F$ with $\lk_{\Delta}(F)\subseteq 2^B$. Then
  $\widetilde{H}_{d-1-\# F}(|\Delta_B|;\KK)\neq 0$.
\end{lemma}
\begin{proof}
  Let $A$ be the vertex set of $\lk_\Delta(F)$. Since $\Delta$ is a regular triangulation of a ball
  $|\st_\Delta(F)|$ is convex and hence the points from $A$ are in convex position. For a face $G$
  of $\Delta_A$ we have $G \in \st_\Delta(F) = 2^F * \lk_\Delta(F)$. Hence $G \in \lk_\Delta(F)$ and
  $\Delta_A = \lk_\Delta(F)$. Thus in this case the assertion follows from the fact that the assumptions
  imply  $\widetilde{H}_{d-1-\# F}(\lk_\Delta(F);\KK)= \KK$.
  We know from \ref{lem:homotopy} that $|\Delta_A| = |\lk_\Delta(F)|$ is a deformation retract of
  $|\Delta_{\Omega \setminus F}|$. Thus the inclusion
  $\lk_\Delta(F) \hookrightarrow \Delta_{\Omega \setminus F}$ induces a map in homology that
  sends the generator of $\widetilde{H}_{d-1-\# F}(\lk_\Delta(F);\KK)= \KK$ identically to the
  generator of $\widetilde{H}_{d-1-\# F}(\Delta_{\Omega \setminus F};\KK)= \KK$. In particular, if we
  choose a $(d-1 - \# F)$-cycle $\sigma$ representing the homology class
  $\widetilde{H}_{d-1-\# F}(\lk_\Delta(F);\KK)= \KK$, then $\sigma$ also represents the homology class
  $\widetilde{H}_{d-1-\# F}(\Delta_{\Omega \setminus F};\KK)= \KK$. In particular, $\sigma$ is not a
  boundary in $\Delta_{\Omega \setminus F}$. Let $A \subseteq B \subseteq \Omega \setminus F$. Then
  $B$ supports $\sigma$ and by $\Delta_B \subseteq \Delta_{\Omega \setminus F}$ it cannot be a boundary
  in $\Delta_B$. In particular, we have $\widetilde{H}_{d-1-\# F}(|\Delta_B|;\KK)\neq 0$.
\end{proof}

We can now use the previous two lemmas to derive the main result of this section. In order to apply these
lemmas we need to construct faces $F\in \Delta_{d-1}^{\langle r \rangle}$ satisfying the property that
$\partial |F| =|F|\cap  \partial |\Delta|$.

\begin{lemma}\label{lem:face}
  Let $r\geq d$. For $0\leq j\leq d-2$, let $v^{(j)}\in \Omega_{r-1,d-j-1}$ such that all coordinates
  are greater than $0$. Let
  \begin{equation*}
     F_j:= \{(\underbrace{1,0,\ldots,0}_{j-1},v^{(j)}),(\underbrace{0,1,0,\ldots,0}_{j-1},v^{(j)}),\ldots,
             (\underbrace{0,\ldots,0,1}_{j-1},v^{(j)})\}.
  \end{equation*}
  Then $F_j\in \Delta_{d-1}^{\langle r\rangle}$ and
  $\partial|F_j|=|F_j|\cap\partial|\Delta_{d-1}^{\langle r\rangle}|$.
\end{lemma}

\begin{proof}
  Since $v^{(j)}\in \Omega_{r-1,d-j-1}$ it follows that $F_j\in  \Delta_{d-1}^{\langle r\rangle}$.
  Moreover, for any $j$-tuple of vertices there exists $1\leq k\leq d-j-1$ such that their
  $k$\textsuperscript{th} coordinates equal $0$. Hence they lie on the hyperplane $x_k=0$ and
  therefore on a facet of $\Delta_{d-1}^{\langle r\rangle}$. This shows $\partial|F_j|\subseteq \partial |\Delta|$.
  Moreover any point in the interior of $|F_j|$ lies in the interior of
  $|\Delta_{d-1}^{(r)}|$ since $v^{(j)}$ has only non-zero coordinates.
\end{proof}

Now we state the main result of the section which improve on the bounds provided in \ref{Prop:edgewise1}.

\begin{theorem}\label{edgewise:Betti}
 Let $r \geq d$. Then
    \begin{itemize}
    \item[(i)] If $1\leq j\leq \frac{d}{2}$, then
      $$
        \beta_{i,i+j}(\KK[\Delta_{d-1}^{\langle r \rangle}])
        \begin{cases}
                 =0 \mbox{ for } 0\leq i\leq j-1,\\
             \neq 0 \mbox{ for } j \leq i\leq \pdim(\KK[\Delta_{d-1}^{\langle r \rangle}]).
        \end{cases}$$
   
    \item[(ii)] If $\frac{d}{2}<j\leq d-2$, then
      $$\beta_{i,i+j}(\KK[\Delta_{d-1}^{\langle r \rangle}]) 
	\begin{cases}
	=0 \mbox{ for } 0\leq i\leq j-1,\\
	\neq 0 \mbox{ for }  m_j \leq j \leq \pdim (\KK[\Delta_{d-1}^{\langle r \rangle}]).
      \end{cases}$$
    
    \item[(iii)] For $j=d-1$  we have 
    $$\beta_{i,i+j}(\KK[\Delta_{d-1}^{\langle r \rangle}])\neq 0 \mbox{  for }  2^d-1-d\leq i\leq \pdim (\KK[\Delta_{d-1}^{\langle r\rangle}]).$$

  \end{itemize}

\end{theorem}

\begin{proof}
The statements concerning the vanishing of Betti numbers in (i) and (ii) follow from the same arguments as the corresponding statements in the proof of \ref{le:barysimplex} since $\Delta_{d-1}^{\langle r\rangle}$ is a flag complex. 

(i) By \ref{Prop:edgewise1} it suffices to show that $\beta_{i,i+j}(\KK[\Delta_{d-1}^{\langle r \rangle}]) \neq 0$ for 
  $2^d-d-1-m_{d-1-j}<i\leq \pdim(\KK[\Delta_{d-1}^{\langle r \rangle}])=\#\Omega_{r,d}-d$. Let $F_j$ as in \ref{lem:face} 
  and let $A_j$ be the vertex set of $\lk_{\Delta_{d-1}(r)}(F_j)$. It follows from \ref{lem:nonvanishing} that 
  $\widetilde{H}_{d-1-j-1}(\Delta_{d-1}^{\langle r\rangle})_{B_j};\KK)\neq 0$ for 
  $A_j\subseteq B_j\subseteq \Omega_{r,d}\setminus F_j$. By Hochster's formula \ref{hochster} we conclude 
  $\beta_{i,i+d-j-1}(\KK[\Delta_{d-1}^{\langle r \rangle}]) \neq 0$ for 
  $\#A_j-(d-j-1)\leq i\leq \#\Omega_{r,d}-(j+1)-(d-j-1)=\pdim (\KK[\Delta_{d-1}^{\langle r \rangle}])$.
  By \ref{lem:InnerFace} we know that  $\lk_{\Delta_{d-1}(r)}(F_j)$ is abstractly isomorphic to the boundary of the 
  barycentric subdivision of a $(d-1-j)$-simplex. Hence $\#A_j=2^{d-j}-2$. It remains to show that 
  $2^{d-j}+j\leq 2^d-m_{d-1-j}$ or equivalently $m_{d-1-j}\leq 2^d-2^{d-j}-j$. From \ref{eq:BettiFirstPart} we can 
  infer that $m_{d-1-j}\leq 2^{d-j}+j-d-1$ if $1\leq j\leq \frac{d}{2}$. Since it can be easily verified that $2^{d-j}+j-d-1\leq 2^d-2^{d-j}-j$ for $1\leq j\leq \frac{d}{2}$, the claim follows.
  The proofs of (ii) and (iii) use a similar reasoning.
\end{proof}

\subsection{Asymptotic behavior of Betti numbers for edgewise subdivisions}

The focus in this section lies on the study of Betti numbers of edgewise subdivisions of arbitrary
simplicial complexes, i. e., we do not restrict our attention to simplices anymore.
Similar as we did for iterated barycentric subdivisions, as $r$ grows, we want to determine the relative
amount of non-zero Betti numbers $\beta_{i,i+j}(\KK[\Delta^{\langle r \rangle}])$, for $j$ fixed, compared to the
projective dimension of $\KK[\Delta^{\langle r \rangle}]$. More precisely, given $1\leq j\leq d-1$, we study the quantity

\begin{equation*}
  \frac{\# \{i~:~\beta_{i,i+j}(\KK[\Delta^{\langle r \rangle}])\neq 0\}}{\pdim( \KK[\Delta^{\langle r \rangle}])}.
\end{equation*}

Our main result in this section, paralleling \ref{thm:asymptoticsBary}, is the following:

\begin{theorem}
\label{thm:asymptoticsEdge}
  Let $d-1\geq 1$ and let $\Delta$ be a $(d-1)$-dimensional simplicial complex and let $r\geq 2d$ be a positive integer.
  Then, $\beta_{i,i+j}(\KK[\Delta^{\langle r \rangle}])\neq 0$ in the following cases:
  \begin{itemize}
    \item[(i)] $1\leq j\leq \frac{d}{2}$ and $j\leq i\leq \binom{2d-1}{d-1}-d+\pdim (\KK[\Delta^{\langle r \rangle}])+\depth (\KK[\Delta])-\binom{3d-1}{d-1}.$
    \item[(ii)] $\frac{d}{2}\leq j\leq d-2$ and $m_j\leq i\leq \binom{2d-1}{d-1}-d+\pdim (\KK[\Delta^{\langle r \rangle}])+\depth (\KK[\Delta])-\binom{3d-1}{d-1}.$
    \item[(iii)] $j=d-1$ and $2^d-d-1\leq i\leq  \binom{2d-1}{d-1}-d+\pdim (\KK[\Delta^{\langle r \rangle}])+\depth (\KK[\Delta])-\binom{3d-1}{d-1}$.
  \end{itemize}
\end{theorem}

The proof of this theorem follows a similar strategy as the proof of \ref{thm:asymptoticsBary}.
The main idea is to isolate a subcomplex $\Gamma$, that is isomorphic to the
$d$\textsuperscript{th} edgewise subdivision of a $(d-1)$-simplex from the rest of the complex
$\Delta^{\langle r \rangle}$. Then we apply \ref{edgewise:Betti} to $\Gamma$. Eventually, one uses that any
homologically non-trivial induced subcomplex of $\Gamma^{\langle r \rangle}$ give rise to homologically non-trivial
induced subcomplexes of $\Delta^{\langle r \rangle}$ when adding vertices not connected to $\Gamma^{\langle r\rangle}$. We make this
more precise in the following proof.

\begin{proof}
  Since $\Delta$ is of dimension $d-1$, we can choose a $(d-1)$-dimensional face $F$ of $\Delta$.
  If $r\geq 2d$, there
  exist subcomplexes of $(2^F)^{\langle r \rangle}$, that are isomorphic to $\Delta_{d-1}^{\langle 2d\rangle}$.
  Let $\widetilde{\Delta}$ be one of those subcomplexes. Note, that there exists more than one such
  subcomplex as soon as $r>2d$. In $\widetilde{\Delta}$ there is another subcomplex $\Gamma$ that is
  isomorphic to $\Delta_{d-1}^{\langle d \rangle}$ and that completely lies in the interior of $\widetilde{\Delta}$.
  Just peel off the ``outer'' layers of $\widetilde{\Delta}$. In particular, there does not exist any
  edge in $\Delta^{\langle r \rangle}$ that connects some vertex in $\Gamma$ to some vertex in
  $\Delta^{\langle r \rangle}\setminus \widetilde{\Delta}$.
  Let $V_{\Gamma}$, $V_{\widetilde{\Delta}}$ and $V_{\Delta}$ denote the vertex sets of $\Gamma$,
  $\widetilde{\Delta}$ and $\Delta^{\langle r \rangle}$, respectively. By construction, we have
  $\Gamma=(\Delta^{\langle r \rangle})_{V_{\Gamma}}$ and thus $\Gamma$ is an induced subcomplex of $\Delta^{\langle r \rangle}$.
  Arguing as in the proof of \ref{thm:asymptoticsBary} one shows that
  \begin{equation*}
    \widetilde{H}_j(\Delta^{\langle r \rangle}_{A\cup B};\KK)\neq 0
  \end{equation*}
  for any $A\subseteq V_{\Gamma}$ with $\widetilde{H}_j(\Gamma_A;\KK)\neq 0$ and any
  $B\subseteq V_{\Delta}\setminus V_{\widetilde{\Delta}}$. Let $l_{\Gamma}(j)$ and $u_{\Gamma}(j)$ denote the beginning and
  the end of the $j$\textsuperscript{th} strand of the resolution of $\KK[\Gamma]$. Using Hochster's formula 
  \ref{hochster} we conclude that $\beta_{i,i+j}(\KK[\Delta^{\langle r \rangle}])\neq 0$ for
  $l_{\Gamma}(j)\leq i\leq u_{\Gamma}(j)+\# V_{\Delta}-\# V_{\widetilde{\Delta}}$. Let $1\leq j\leq \frac{d}{2}$.
  In this case, \ref{thm:asymptoticsBary} (i) implies that  $\beta_{i,i+j}(\KK[\Delta^{\langle r \rangle}])\neq 0$ if
  \begin{align*}
     j\leq i & \leq \pdim (\KK[\Delta_{d-1}^{\langle d\rangle}])+\pdim (\KK[\Delta^{\langle r \rangle}])+\depth( \KK[\Delta^{\langle r\rangle}])-\# V_{\widetilde{\Delta}}\\
             &=\binom{2d-1}{d-1}-d+\pdim (\KK[\Delta^{\langle r \rangle}])+\depth (\KK[\Delta])-\binom{3d-1}{d-1},
  \end{align*}
  where we use that $\Delta_{d-1}^{\langle r\rangle}$ has $\binom{d+r-1}{d-1}$ many vertices and that the depth is invariant under taking edgewise subdivisions (see Table \ref{How}).  
  This shows (i).
  We omit the proofs of (ii) and (iii) since they follow from similar reasoning, using parts (ii) and (iii)
  of \ref{Prop:edgewise1}.  
  
\end{proof}
As an immediate consequence of \ref{thm:asymptoticsEdge} and of the fact that $\lim_{r\rightarrow \infty} \pdim (\KK[\Delta^{\langle r \rangle}])=\infty$  we obtain: 

\begin{corollary}
\label{coro:asymptoticsEdge}
  Let $d-1\geq 1$ and let $\Delta$ be a $(d-1)$-dimensional simplicial complex. For $1\leq j\leq d-1$ one has that $\#\{ i : \beta_{i,i+j}(\KK[\Delta^{\langle r \rangle}]= 0\}$ is bounded above in terms of $j $ and $d$. In particular
    \begin{equation*}
    \lim_{r\rightarrow \infty} \frac{\# \{i~:~\beta_{i,i+j}(\KK[\Delta^{\langle r \rangle}])\neq 0\}}{\pdim (\KK[\Delta^{\langle r \rangle}])}=1.
  \end{equation*}
  for every $j=1,\dots,d-1$. 
\end{corollary}

As for iterated barycentric subdivision, the above theorem does not cover the last strand of the resolution of the edgewise subdivision, in case the regularity is $d$ or equivalently when the simplicial complex has homology in top-dimension. It turns out, that in this
setting the asymptotic behavior of the last strand of the resolution depends on the simplicial
complex. We now recall a special case of Theorem 5.1 from \cite{Diaconis}.

\begin{proposition}\label{cor:edgewiseF}
  Let $\Delta$ be a $(d-1)$-dimensional simplicial complex. Then, as $r\rightarrow \infty$,
  \begin{equation*}
    \frac{f_0^{\Delta^{\langle r \rangle}}}{r^{d-1}}\quad \quad \rightarrow \frac{f_{d-1}^{\Delta}}{(d-1)!}.
  \end{equation*}
\end{proposition}

A more simple fact paralleling \ref{lem:minimal} can be easily deduced from the $f$-vector transformation
in edgewise subdivisions (see \cite{BW-Veronese}).

\begin{lemma} \label{lem:minimale}
Let $\Delta,\Delta'$ be two $(d-1)$-dimensional simplicial complexes such that for some 
$0 \leq i \leq d-1$ we have $f_i^\Delta > f_i^{\Delta'}$ and $f_{j}^\Delta = f_j^{\Delta'}$ for
$i < j \leq d-1$. Then there exists $R$ such that for $r \geq R$ we have 
  \begin{eqnarray*}
     f_j^{\Delta^{\langle r \rangle}}>f_j^{\Delta'^{\langle r \rangle}} \mbox{~for~} 0 \leq j \leq i \\
     f_j^{\Delta^{\langle r \rangle}}=f_j^{\Delta'^{\langle r \rangle}} \mbox{~for~} i < j \leq d-1.
  \end{eqnarray*}
\end{lemma}

These results are crucial for the next proposition, which treats the asymptotics of the last strand of
the resolution of $\KK[\Delta^{\langle r \rangle}]$ if the $(d-1)$-dimensional simplicial complex $\Delta$ has
homology in top dimension, i.e., if $\reg (\KK[\Delta^{\langle r \rangle}])=d$.

\begin{proposition}
\label{prop:edgewiseHomology}  
  Let $d-1\geq 1$ and let  $\Delta$ be a $(d-1)$-dimensional simplicial complex such that
  $\widetilde{H}_{d-1}(\Delta;\KK)\neq 0$. Let further $\sigma$ be a minimal homology
  $(d-1)$-cycle of $\Delta$ and let 
  \begin{equation*}
    \widetilde{\sigma}=\{F\in \Delta~:~F\subseteq G\mbox{ for some } G \mbox{ in the support of } \sigma\}
  \end{equation*}
be the corresponding induced subcomplex of $\Delta$, whose vertex set is $V_r^{\sigma}$.
  Then
  \begin{itemize}
    \item[(i)] for $r\geq 1$ 
      $\beta_{i,i+d}(\KK[\Delta^{\langle r \rangle}])\neq 0$ for $ \# V_r^{\sigma}-d\leq i\leq \pdim (\KK[\Delta^{\langle r \rangle}])$. If $r$ is large, then in addition $\beta_{i,i+d}(\KK[\Delta^{\langle r \rangle}])= 0$ for $0\leq i< \# V_r^{\sigma}-d$.
    \item[(ii)]
      \begin{equation*}
        \lim_{r\rightarrow\infty}\frac{\# \{i~:~\beta_{i,i+d}(\KK[\Delta^{\langle r \rangle}])\neq 0\}
}{\pdim \KK[\Delta^{\langle r \rangle}]}           = 1-\frac{f_{d-1}^{\widetilde{\sigma}}}{f_{d-1}^{\Delta}}
      \end{equation*}
  \end{itemize}
\end{proposition}

\begin{proof}
   Since $(\Delta^{\langle r \rangle})_{V_r^{\sigma}}=\widetilde{\sigma}^{\langle r \rangle}$, and since $\sigma$ is a homology $(d-1)$-cycle,
   it holds that $\widetilde{H}_{d-1}((\Delta^{\langle r \rangle})_{V_r^{\sigma}};\KK)\neq 0$. Thus, using Hochster's formula \ref{hochster},
   we infer $\beta_{i,i+d}(\KK[\Delta^{\langle r \rangle}])\neq 0$ for $i=\# V_r^{\sigma}-d$. By the same reasoning as in the
   proof of \ref{prop:BarycentricHomology} we can further conclude that $\beta_{i,i+d}(\KK[\Delta^{\langle r \rangle}])\neq 0$
   for $\# V_r^{\sigma}-d\leq i\leq \pdim (\KK[\Delta^{\langle r \rangle}])$, which shows the non-vanishing in (i).
   For the vanishing we use the minimality of $\sigma$ and \ref{lem:minimale} in the same way as in the
   proof of \ref{prop:BarycentricHomology}.

   We now prove (ii). Let $V_r^{\Delta}$ denote the vertex set of $\Delta^{\langle r \rangle}$. It follows from (i) that
   \begin{align*}
     & \empty\frac{1}{\pdim( \KK[\Delta^{\langle r \rangle}])}\# \{i~:~\beta_{i,i+d}(\KK[\Delta^{\langle r \rangle}])\neq 0\}\\
     & = \frac{1}{\pdim (\KK[\Delta^{\langle r \rangle}])}(\pdim( \KK[\Delta^{\langle r \rangle}])-(\# V_r^{\sigma}-d-1))\\
     & = 1 - \frac{\# V_r^{\sigma}-d-1}{\# V_r^{\Delta}-\depth (\KK[\Delta^{\langle r\rangle}])}\\
     & = 1 - \frac{\frac{1}{r^{d-1}}(\# V_r^{\sigma}-d-1)}{\frac{1}{r^{d-1}}(\# V_r^{\Delta}-
                                       \depth (\KK[\Delta]))}\\
                &=1 - \frac{\frac{1}{r^{d-1}}(f_0^{\widetilde{\sigma}^{\langle r\rangle}}-d-1)}{\frac{1}{r^{d-1}}(f_0^{\Delta^{\langle r\rangle}}-
                                       \depth (\KK[\Delta]))}.
   \end{align*}
   As $r$ goes to infinity, \ref{cor:edgewiseF} implies that
   $\frac{\# \{i~:~\beta_{i,i+d}(\KK[\Delta^{\langle r \rangle}])\neq 0\}}{\pdim \KK[\Delta^{\langle r \rangle}]}$  approaches
   $\frac{f_{d-1}^{\widetilde{\sigma}}}{f_{d-1}^{\Delta}}$.
\end{proof}
 
Since the limits in \ref{prop:BarycentricHomology} and \ref{prop:edgewiseHomology} coincide \ref{ex:limit} shows that for any $d$ any rational number in the half-open intervall $[0,1)$ can occur as a limit in \ref{prop:edgewiseHomology}.

\section{Appendix}
\label{Appendix} 
In this section we provide the proofs of the statements that require only   manipulations of simple numerical expressions. 

\begin{proof}[ Proof of \ref{againmj}] 
We set 
\begin{equation} \label{eq:nj}
  n_j=\min\Big\{\sum_{\ell=1}^r (2^{i_\ell+2}-2): \begin{array}{l}
  (i_1,\dots,i_r)\in \NN^r,\\
  i_1+\cdots +i_r+(r-1)=j-1,\\
  i_1+\cdots +i_r+2r\leq d
  \end{array} 
  \Big\}-j.
\end{equation}
 
 Our goal is to prove that $n_j=m_j$ where the $m_j$'s are defined in \ref{defmj}. 
We may write the  second condition in \ref{eq:nj}  as 

 \begin{equation}
 \label{App1} 
  i_1+\cdots +i_r=j-r
  \end{equation} 
and the  third  as 
$$j+r\leq d,$$
or, equivalently, 
\begin{equation}
 \label{App2}  r\leq d-j.
 \end{equation} 
 
 Note  also the   \ref{App1}  implies that $r\leq j$ since $(i_1,\dots,i_r)\in \NN^r$. Therefore $r\leq \min\{j, d-j\}$. 
Summing up,  we can rewrite the definition of $n_j$ as 

\begin{equation} \label{eq:nj2}
n_j=\min_{r=1,\dots,\min\{j,d-j\}} \min \Big\{\Big(\sum_{\ell=1}^r  2^{i_\ell+2}\Big ) -2r~: 
\begin{array}{l}
(i_1,\dots,i_r)\in \NN^r, \\
 i_1+\cdots +i_r=j-r
 \end{array}  \Big\}-j. 
 \end{equation} 

Observe the following: 

\noindent {\bf Claim 1}:   Assume  $r<\min\{j,d-j\}$, $i=(i_1,\dots,i_r)\in \NN^r$ and $i_1+\cdots +i_r=j-r$.  Note that $j-r>0$ and hence  one of the $i_k$ is positive, say $i_1>0$.  Then we set $u=(u_1,\dots,u_{r+1})$ where $u_1=i_1-1$, $u_k=i_k$ for $k=2,\dots,r$ and $u_{r+1}=0$. By construction, 
$$u_1+\dots+u_{r+1}=i_1+\cdots +i_r-1=j-(r+1)$$ 
and hence $u$ is a ``valid" vector. We want to show that the ``contribution" of $u$ is strictly smaller than that  of the vector $i$. Hence we consider the difference of the ``contributions" of the vector  $i$ and the vector  $u$. 
$$\Big(\sum_{l=1}^r  2^{i_l+2}\Big ) -2r - \Big( \Big(\sum_{l=1}^{r+1}  2^{u_l+2}\Big ) -2(r+1)\Big),$$ 
which is 
$$2^{i_1+2}-2r-2^{u_1+2}-2^{u_{r+1}+2} +2(r+1),$$
that is 
$$2^{i_1+2}-2^{i_1+1}-2^2+2,$$
that is 
$$2^{i_1+2}-2^{i_1+1}-2,$$
which is clearly positive  (since $i_1>0$).  
 \medskip 

\noindent {\bf  Claim 2}: 
 Now suppose that  $i_1+\cdots +i_r=j-r$ and that there are two of the $i_k$ whose difference is $>1$, say $i_2-i_1>1$.   
 Now we define $u=(u_1,\dots,u_r)$ by $u_1=i_1+1, u_2=i_2-1$ and $u_k=i_k$ for $k>2$. We want to show that the contribution of 
 $u$ is smaller than the one of $i$. Hence we consider the difference of the contributions and we have: 
 $$2^{i_1+2}+2^{i_2+2}-2^{i_1+3}-2^{i_2+1}.$$
 We want to show that  it is positive. We may factor out $2^{i_1+2}$ and we have to show that 
 $$1+2^{i_2-i_1}-2-2^{i_2-i_1-1}$$
 is positive, that is, 
 $$2^{i_2-i_1}-1-2^{i_2-i_1-1}$$
is positive.  This is clearly true. 
\medskip  

Taking Claim 1 and Claim 2 into consideration we have that the minimum in the expression of \ref{eq:nj2} for $n_j$ is obtained when $r=\min\{j,d-j\}$ and for the  vector $(i_1,\dots,i_r)$ such that $i_1\geq i_2\geq  \dots \geq i_r$ and $i_1-i_r\leq 1$. 
Now if $j\leq d-j$, i.e., $j\leq d/2$, then $r=j$  and   the corresponding vector is $(i_1,\dots,i_j)=(0,\dots, 0)$. It follows then that $n_j=j$. 

If instead $d-j\leq j$, i.e., $j\geq d/2$, then $r=d-j$ and the  corresponding vector $(i_1,\dots,i_{d-j})$ is obtained as follows. 
Since $$i_1+\dots+i_{d-j}=j-(d-j)=2j-d$$
and $i_1\geq i_2\geq  \dots \geq i_{d-j}$ and $i_{1}-i_{d-j}\leq 1$  it must hold that 
$$i_k=\left\{ \begin{array}{ll}
a+1 & \quad\mbox{ for }k=1,\dots, c\\
a     &\quad\mbox{ for } k=c+1,\dots, d-j
\end{array} 
\right.
$$ 
where $a$ and $c$ are non-negative integers such that 
$$(2j-d)=a(d-j)+c$$
and $0\leq c< d-j$. 
Hence we have
 \begin{align*}
n_j&= \Big(\sum_{\ell=1}^{d-j}   2^{i_\ell+2}\Big ) -2(d-j)-j=c 2^{a+3}+(d-j-c)2^{a+2} -2d+j \\
&=  2^{a+2} (2c+d-j-c)-2d+j=2^{a+2} (c+d-j)-2d+j
 \end{align*}
 
Hence we have shown that $n_j$ is equal to $m_j$. 
 \end{proof} 

\begin{proof}[ Proof of \ref{lem:i1=0}]
  We have the following chain of inequalities:
  \begin{align*}
    (2^{i_r+2}-2)2^{i_2+\cdots +i_{r-1}+2r-4}&=2^{i_r+2}\left(1-\frac{1}{2^{i_r+1}}\right)2^{i_2+\cdots +i_{r-1}+2r-4}\\
       &\geq 2^{i_r+2}\left(1-\frac{1}{2}\right)2^{i_2+\cdots +i_{r-1}+2r-4}\\
       &=2^{i_r+1+i_2+\cdots +i_{r-1}+2r-4}.
  \end{align*}
  Since $i_1=0$ by (i), it follows from (ii) that $i_2+\cdots +i_{r-1}+i_r+(r-1)=j-1$ and hence,
  \begin{equation*}
    (2^{i_r+2}-2)2^{i_2+\cdots +i_{r-1}+2r-4}\geq 2^{j+r-3}.
  \end{equation*}
  For $r\geq 4$ the claim follows. We need to treat the cases $r\in \{1,2,3\}$ separately.

  \noindent {\sf Case 1:} $r=1$.
    By (ii) it then follows that $j=1$ and hence $d<2$, which is a contradiction to the assumptions.

  \noindent {\sf Case 2:} $r=2$. It follows from (ii) that $i_2=j-2$ and hence
    \begin{align*}
      &\empty(2^{i_r+2}-2)2^{i_2+\cdots +i_{r-1}+2r-4}+\sum_{l=1}^r(2^{i_l+2}-2)\\
      &= (2^j-2)+(2^{0+2}-2)+(2^{j}-2)\\
      &=2\cdot 2^j-2=2^{j+1}-2,
    \end{align*}
    which shows the claim.

  \noindent {\sf Case 3:} $r=3$. By (i) and (ii) it holds that $i_2+i_3=j-3$. From this, we obtain
    \begin{align*}
      &\empty(2^{i_r+2}-2)2^{i_2+\cdots +i_{r-1}+2r-4}+\sum_{l=1}^r(2^{i_l+2}-2)\\
      &=(2^{i_3+2}-2)2^{i_2+2}+(2^{0+2}-2)+(2^{i_2+2}-2)+(2^{i_3+2}-2)\\
      &=2^{i_2+i_3+4}-2^{i_2+3}+2^{i_2+2}+2^{i_3+2}-2\\
      &=2^{j+1}-2^{i_2+2}+2^{i_3+2}-2\\
       &\geq 2^{j+1}-2,
    \end{align*}
    where for the last inequality we use $i_2\leq i_3$.
\end{proof}

\begin{proof}[ Proof of \ref{lem:i1>1}] 
  Since
  \begin{align*}
    \sum_{\ell=1}^r(2^{i_{\ell}+2}-2)-
    \sum_{\ell=1}^r(2^{j_{\ell}+2}-2)&=2^{i_1+2}+2^{i_2+2}-2^{j_1+2}-2^{j_2+2}\\
       &=2^{i_1+2}+2^{i_2+2}-2^{i_1+1}-2^{i_2+3}=2^{i_1+1}-2^{i_2+2},
  \end{align*}
  we need to show that
  \begin{equation}\label{eq1}
    (2^{i_r+2}-2)2^{i_2+\cdots +i_{r-1}+2r-4}+2^{i_1+1}-2^{i_2+2}\geq 0.
  \end{equation}
  We have
  \begin{align*}
    (2^{i_r+2}-2)2^{i_2+\cdots +i_{r-1}+2r-4}&= 2^{i_r+2}\left(1-\frac{1}{2^{i_r+1}}\right)2^{i_2+\cdots +i_{r-1}+2r-4}\\
    &\geq  2^{i_r+2}\left(1-\frac{1}{2}\right)2^{i_2+\cdots +i_{r-1}+2r-4}\\
    &=2^{i_r+1}2^{i_2+\cdots +i_{r-1}+2r-4}=2^{i_2+\cdots +i_{r-1}+i_r+2r-3}.
  \end{align*}
  For $r\geq 3$ the last expression is $\geq 2^{i_2+i_r+6-3}\geq 2^{i_2+2}$ from which follows \ref{eq1} in
  this case. If $r=2$, we obtain
  \begin{equation*}
    (2^{i_r+2}-2)2^{i_2+\cdots +i_{r-1}+2r-4}=2^{i_2+2}-2.
  \end{equation*}
  Hence, \ref{eq1} also holds in this case.
\end{proof}

\section{Acknowledgement*} 
We thank Daniel Erman and Svante Linusson for useful discussions and for pointing us to results from the literature
that are needed in our arguments.

\bibliographystyle{plain}
  
\end{document}